\documentclass[11pt]{amsart}

\usepackage{amsmath,amssymb}
\usepackage{mathrsfs} 
\usepackage{enumerate}
\usepackage{appendix}
\usepackage{wrapfig}

\usepackage{tablefootnote}

\usepackage{pgf,tikz}
\usetikzlibrary{arrows}
\usepackage[all]{xy}
\usepackage{appendix}

\usepackage{float}

\usepackage{array,booktabs,arydshln,xcolor}

\usetikzlibrary {positioning}
\usetikzlibrary{arrows}
\usetikzlibrary{positioning}
\usepackage{multirow}
\usepackage{tikz-cd}

    \usetikzlibrary{patterns.meta}
\tikzstyle{stuff_fill}=[rectangle,fill=white,minimum size=1em]

\usetikzlibrary{decorations.pathreplacing}

\usepackage{multicol}

\usepackage{geometry} 
\usepackage{tikz-cd}
\tikzcdset{arrow style=tikz,diagrams={>=stealth}}

\RequirePackage{hyperref}
\hypersetup{pdfpagemode={UseOutlines},
bookmarksopen=true,
bookmarksopenlevel=0,
hypertexnames=false,
colorlinks=true, 
citecolor=teal, 
linkcolor=blue, 
urlcolor=magenta, 
pdfstartview={FitV},
unicode,
breaklinks=true,
}

\usepackage{array,multirow}

\newtheorem{theorem}{Theorem}[section]
\newtheorem{theoremintro}{Theorem}
\newtheorem{corollaryintro}[theoremintro]{Corollary}
\newtheorem{propintro}[theoremintro]{Proposition}
\newtheorem{lemma}[theorem]{Lemma}
\newtheorem{proposition}[theorem]{Proposition}
\newtheorem{corollary}[theorem]{Corollary}

\newtheorem*{theorem*}{Theorem}
\newtheorem*{ques*}{Question}
\newtheorem*{prop*}{Proposition}

\theoremstyle{definition}
\newtheorem{definition}[theorem]{Definition}
\newtheorem{example}[theorem]{Example}

\newtheorem{ques}[theorem]{Question}

\newtheorem*{definition*}{Definition}

\theoremstyle{remark}
\newtheorem{remark}[theorem]{Remark}

\numberwithin{equation}{section}

\renewcommand{\thesubsection}{\arabic{section}.\Alph{subsection}}

\title[Lower bounds on Dehn functions and QI rigidity of SOL\textsubscript{5}]{Dehn functions: computations, lower bounds, and the quasiisometric rigidity of SOL\textsubscript{5}}

\author{Ido Grayevsky}
\address{School of Mathematics, University of Bristol, UK}
\email{grayevsky@gmail.com}

\author{Gabriel Pallier}
\address{Univ. Lille, CNRS, UMR 8524 - Laboratoire Paul Painlevé, F-59000 Lille, France}
\email{gabriel.pallier@univ-lille.fr}

\thanks{The first author was partially funded by Israel Science Foundation grants ISF 2919/19, ISF 2990/21, ISF 1577/23, and the Royal Society grant NIF\textbackslash R1\textbackslash 242363. The second author was funded by the DFG 281869850 (RTG 2229). We gratefully acknowledge the financial support from the ConYS programme of the Karlsruhe House for Young Scientists at the KIT, and from the PEPS JCJC programme of the INSMI. The authors thank the Isaac Newton Institute for Mathematical Sciences for the support and hospitality during the programme {\em Operators, Graphs, Groups} when work on this paper was undertaken. This work was supported by EPSRC, Grant number EP/Z000580/1.  }

\thanks{}


\date{\today}
\dedicatory{}

\begin{document}

\begin{abstract}
We establish distortion estimates in completely solvable Lie groups, using a sublinear bilipschitz retraction constructed by Cornulier, and interpolating between two theorems of Osin.
This provides new lower bounds on Dehn functions. Our second main result is the quasiisometric rigidity of $\operatorname{Sol_5}$ and its lattices. Together with a theorem of Peng, a key tool for the rigidity is the complete list of Dehn functions and dimensions of asymptotic cones of all simply connected solvable Lie groups of exponential growth up to dimension $5$, which we compute using Cornulier and Tessera's results.
\end{abstract}
\maketitle

\tableofcontents

\section{Introduction}
This study is motivated by the quasiisometric classification and rigidity of connected Lie groups. By a result of \cite{CornulierCones11}, this problem amounts to the quasiisometric classification and rigidity in the class of completely solvable groups, which are the closed subgroups of the upper triangular real matrix groups\footnote{These are also called real-triangulable, or split-solvable, in the litterature.}. Our focus here is on Dehn functions, a prominent quasiisometric invariant. In~\cite{CoTesDehn}, Cornulier and Tessera develop a rich theory for computing Dehn functions of completely solvable groups. Their work is remarkable in its completeness: for example, it allows us to determine the Dehn functions of all completely solvable groups that are of exponential growth and dimension up to 5 (see Proposition~\ref{PropIntro: Dehn functions of dim 4 5}).

This paper has two main novel contributions. We develop a theoretical tool that allows one to obtain lower bounds on the Dehn functions using distortions in central extensions of completely solvable groups. This tool is based on a sublinear bilipschitz retract between a group $G$ and the largest nilpotent quotient of $\rho_1(G)$ (see Section~\ref{sec: completely solvable groups} for details).  Our second contribution is a quasiisometric rigidity result for the group $\operatorname{Sol}_5$, isometric to a particular horosphere in $\mathbb H^2 \times \mathbb H^2 \times \mathbb H^2$. We obtain it by combining our Dehn function computations with a theorem of Peng~\cite{PengCoarseI,PengCoarseII}. 

\begin{remark}
    In a recent and closely related paper~\cite{productSBE}, we give several more applications to our Dehn functions computations. In particular, we combine these computations with our results in~\cite{productSBE} in order to distinguish between several quasiisometry classes of completely solvable groups. 
\end{remark}

\subsection{Completely solvable groups, distortions and Dehn functions}

\begin{definition}[Distortion]
Let $G$ be a simply connected Lie group, and let $X$ be a nonzero element of its Lie algebra $\mathfrak g$.
    The distortion function of a one-parameter subgroup $L= \{\exp(tX)\}$ in $G$ is the growth type of the function
    \[ \Delta_L^G(r) = \sup \{t: \exp tX \in B_G(r)    \}. \]
\end{definition}

\begin{definition}[Central depth]
Given $X \in \mathfrak s$ a nonzero vector in a Lie algera $\mathfrak s$, we call central depth of $X$ and write  
    $c_X = \sup\{  j \in \mathbf Z_{\geqslant 1} \cup \{\infty \} \colon X\in C^j \mathfrak s \}$, where $C^1\mathfrak s = \mathfrak s$ and $C^{j+1}\mathfrak s = [\mathfrak s, C^j \mathfrak s]$ for all $j \geqslant 1$.
\end{definition}

The theorem below relates distortion and central depth, interpolating between the main result of \cite{Osindistort} (which corresponds to the case when the group $S$ in the statement is nilpotent and central depths are all finite) and that of \cite{OsinExprad}; we use both theorems of Osin in the proof.

\begin{theoremintro}[Evaluating distortion in completely solvable groups]\label{thm:distort}
    Let $S$ be a completely solvable group with Lie algebra $\mathfrak s$. 
    Let $X \in \mathfrak s$ be nonzero.
    Let $L$ be the one-parameter subgroup generated by $X$ in $S$.
    \begin{enumerate}
        \item
        If $c_X < \infty$, and if $X$ is in a Cartan subalgebra (e.g. if $X$ is a regular or central element) then 
        \[ \Delta_L^S(r) \asymp  r^{c_X}. \]
        \item \cite{OsinExprad}\footnote{By \cite{OsinExprad} the converse of (2) in Theorem~\ref{thm:distort} is also true.} If $c_X =\infty$ then $\Delta_L^S$ is exponential.
    \end{enumerate}
\end{theoremintro}

Distortion estimates of central extensions are known to have consequences for lower bounds on the Dehn functions in certain classes of groups. We adapt this principle here for polynomial lower bounds in simply connected Lie groups and deduce from Theorem~\ref{thm:distort} the following:

\begin{corollaryintro}\label{corIntro: distorion central extensions}
    Let $G$ be a completely solvable group with Lie algebra $\mathfrak{\mathfrak g}$. 
    Let
    \[ 0 \to \mathbf R \to \mathfrak s \overset{\pi}{\longrightarrow} \mathfrak g \to 0\]
    be a central extension.
    Let $X$ be a generator of $\ker \pi$ and assume it has finite central depth. Then $n^{c_X} \preccurlyeq \delta_G(n)$.
\end{corollaryintro}

\begin{remark}
    In the notations of Corollary~\ref{corIntro: distorion central extensions}, if $c_X=\infty$ then $\delta_G(n)$ is at least exponential; this is a special case of~\cite[Theorem 11.C.1]{CoTesDehn}.
\end{remark}

\begin{example}\label{ExampleIntro: Abels group}(See Section \ref{app: Example}).
    Let $G$ be a central product of Abels' second group \cite[Example 5.7.4]{Abels87} and a model filiform nilpotent group of class $3$.
    By Corollary \ref{corIntro: distorion central extensions} the Dehn function of $G$ is at least cubic; precisely one can show that
    \[ n^3 \preccurlyeq \delta_G(n) \preccurlyeq n^4 \log^8 (n) .\]
    The upper bound is due to Cornulier-Tessera \cite{CoTesDehn}. 
\end{example}

The previously known lower bound is due to Cornulier-Tessera; after making some estimates explicit, it is $n^3/\log^8(n)$, as a consequence of the following theorem.

\begin{theorem}[After Cornulier and Tessera {\cite{CoTesDehn}}, see Section~\ref{sec: comparing lower bounds}]\label{Introthm: Dehn bounds via rho_1}\label{thm: Dehn bounds via rho_1}
    Let $G$ be a completely solvable group, $N$ its largest nilpotent quotient. 
Then 
\begin{enumerate}
    \item Either, $G$ has an exponential Dehn function
    \item Or, there exists $e$ (depending on G) such that 
    \begin{equation}
    \label{eq:est-dehn}
        \delta_N(r)/\log^e(r) \preceq \delta_G(r) \preceq  r\cdot  \widehat {\delta_N}(r) \cdot \log^e(r),
    \end{equation}
    where in the right inequality, $\widehat {\delta_N}$ is any regular function larger than $\delta_N$. (A function $f$ is regular if there exists $\alpha>1$ such that $\frac{f(x)}{x^\alpha}$ is non-decreasing on $\mathbf R_{\geqslant 0}$; in all cases where $\delta_N$ is known it is already regular and one may take $\widehat{\delta_N}=\delta_N$.)
\end{enumerate}
Moreover, if $c$ is the nilpotency class of $N$, then one can take $e=2(c+1)$.
\end{theorem}

As Example \ref{ExampleIntro: Abels group} shows, Corollary~\ref{corIntro: distorion central extensions} sometimes allows to get rid of the power of $\log$ factor in the leftmost term of \eqref{eq:est-dehn}.

\begin{remark}
    Cornulier and Tessera's estimates on the Dehn function are sharper than \eqref{eq:est-dehn} in some subclasses of completely solvable groups; see \cite{CoTesDehn} and Section \ref{sec: comparing lower bounds}. 
\end{remark}

\subsection{Quasiisometric rigidity of $\operatorname{Sol}_5$}
The group $\operatorname{Sol}_5$ (also named $G_{5,33}^{-1,-1}$ in \cite{Mubarakzyanov}), is the semidirect product $\mathbf R^2 \ltimes \mathbf R^3$ with diagonal action of the $\mathbf R^2$ torus, such that the three weights $\varpi_1,\varpi_2,\varpi_3 \in \operatorname{Hom}(\mathbf R^2, \mathbf R)$ are linearly independent and sum to zero. Among all the left-invariant metrics on this group, one makes it isometric to a hypersurface 
\[ \mathcal H = \{ (z_1,z_2,z_3) \in \mathbb H^2 \times \mathbb H^2 \times \mathbb H^2 \colon b_1(z_1) + b_2(z_2)+b_3(z_3) =0 \}  \]
where for $i=1,2,3$, $b_i$ is a horofunction only depending on the projection to the $i$\textsuperscript{th} factor. As such, it is a higher-rank generalization of the three-dimensional group $\mathrm{Sol}_3$ (or SOL) which admits the same description as a horosphere in $\mathbb H^2 \times \mathbb H^2$.

\begin{theoremintro}[{Propositions~\ref{prop:qi-rigidity}--\ref{prop:qi-rididity-C0}}]
\label{prop:qi-rigidity-intro}
The following hold: 
\begin{enumerate}
\item 
Let $G$ be a group of class $(\mathcal C_0)$, quasiisometric to $\operatorname{Sol}_5$. Then $G$ is isomorphic as Lie group to $\operatorname{Sol}_5$.
    \item 
    Let $\Gamma$ be a finitely generated group quasiisometric to $\operatorname{Sol}_5$. Then there is a finite-index subgroup $\Gamma_0$ in $\Gamma$ and a homomorphism $\Gamma_0 \to \operatorname{Sol}_5$ with finite kernel and whose image is a lattice in $\operatorname{Sol}_5$.
\end{enumerate}

\end{theoremintro}
 \begin{remark}\label{rem:statement-of-qi-rigidity-intro}
    Passing to the finite index subgroup $\Gamma_0$ of $\Gamma$ is necessary in the quasiisometric rigidity statement  (2) of Theorem~\ref{prop:qi-rigidity-intro} above. See Remark~\ref{Rem: finite-index-subgroup}.
    \end{remark}

    The distinctive properties of $\operatorname{Sol}_5$ used in the proof of Theorem~\ref{prop:qi-rigidity-intro} is that, unlike all the other completely solvable Lie groups of the same dimension and cone-dimension, it is unimodular and has an exactly quadratic Dehn function. The fact that its Dehn function is quadratic was proved by Dru\c{t}u \cite[Theorem 1.1]{DrutuFillingSol} and Leuzinger-Pittet \cite[Corollary 2.1]{LeuzingerPittetQuadratic}. The fact that it is the only such group is a consequence of our computations of Dehn functions, which we sum up in the following proposition. 

\begin{propintro}\label{PropIntro: Dehn functions of dim 4 5}
    The simply connected solvable Lie groups of dimension less or equal five and exponential growth have Dehn function of growth type $n$, $n^2$, $n^3$, $n^4$ or $\exp(n)$ exactly, all given for the groups of dimension less than $5$ in Table~\ref{tab:groupsless5prop} and for the groups of dimension $5$ in Table~\ref{tab:Dehn-functions-5}.
    Moreover, for all these groups, either the Dehn function is exponential, or it is equal to that of the largest nilpotent quotient.
\end{propintro}

Our computations leading to Proposition~\ref{PropIntro: Dehn functions of dim 4 5} use Cornulier and Tessera's theorems in~\cite{CoTesDehn}. 

In relation with Theorem~\ref{Introthm: Dehn bounds via rho_1}, Proposition \ref{PropIntro: Dehn functions of dim 4 5} leads to ask whether the Dehn function of a completely solvable group is always either exponential or equal to that of its largest nilpotent quotient. According to Cornulier and Tessera (work in progress, \cite{CTSecondAbels}), Abels' second group does not satisfy this simple guess: its Dehn function is strictly faster than quadratic, although at most cubic.

More generally, we can reframe this into the language of isoperimetric spectrums. Let us introduce two sets of real numbers:
\[ \mathsf{LIP} = \{ \alpha \in [1,+\infty): \exists G\, \text{connected Lie}, \lim_{n \to + \infty} \frac{\log \delta_G(n)}{\log n}=\alpha \}, \]
\[ \mathsf{NIP_{\mathbf R}} = \{ \alpha \in [1,+\infty): \exists N\, \text{connected Lie nilpotent}, \lim_{n \to + \infty} \frac{\log \delta_N(n)}{\log n}=\alpha \} \]
and note that
\[ 
\mathbf N_{\geq 1} \subseteq \mathsf{NIP}_{\mathbf R} \subseteq \mathsf{LIP} \subseteq \{ 1 \} \cup [2,+\infty). 
\]

\begin{ques}
    Are there equalities in the chain of inclusions above?
\end{ques}

The discussion above on a possible improvement of Theorem~\ref{Introthm: Dehn bounds via rho_1} is now related to knowing whether there is equality in the second inclusion. Ultimately, 
this question is about how diverse the Dehn functions of nilpotent and solvable Lie groups can be, and if there is a difference. For a comparison it is known that the closure of $\mathsf{IP}$, the isoperimetric spectrum of finitely presented groups is equal to $\{ 1 \} \cup [2,+\infty)$ \cite{BBIsop}.

 \begin{remark}
     Notice that multiplicative differences of power of logarithms factors in the Dehn function are not seen in the isoperimetric spectrums.  Combining Theorem~\ref{th:Cornulier-thm} with Corollary~\ref{cor: logSBE and poly Dehn implies subexponential dehn} we see that the set $\mathsf{LIP}$ does not change if one considers, instead of all connected Lie groups, only Lie groups of the class $\mathcal{C}_1$ (see Section~\ref{Sec: class C1} for details). 
 \end{remark}
Finally, our computations raise interest in two families of groups with parameters. These families exhibit interesting behavior, which we present in Section~\ref{sec: particular families}. In particular we discuss their quasiisometric invariants, geometry, and what is missing in order to achieve quasiisometric classification within these families. 

\subsection{Organization of the paper}

Section~\ref{sec: completely solvable groups} introduces the theory of completely solvable Lie groups and the significance of sublinear bilipschitz equivalence (SBE) to this theory. Section~\ref{sec:distortion-dehn} is dedicated to the proof of the lower bounds estimates of Dehn functions. In Section~\ref{sec:proof-D} we prove the distortion estimates stated in Theorem~\ref{thm:distort}, and in Section \ref{sec: distortion in Extension to Dehn} we apply them to obtain the lower bounds. Section~\ref{sec: comparing lower bounds} compares our results with other lower bound methods. In Section~\ref{app: Example} we elaborate on Example~\ref{ExampleIntro: Abels group}, in which our tool improves on previously known techniques.

Sections~\ref{sec: computations and contribution} and~\ref{sec:qi-rigid} contain the concrete contribution towards the quasiisometric classification of low dimensional completely solvable Lie groups. Building on our Dehn function computations,  Section~\ref{sec: particular families} elaborates on particular families of completely solvable Lie groups which exhibit interesting behavior. In Section~\ref{sec:qi-rigid} we prove the quasiisometric rigidity of the group $G_{5,33}^{-1,-1}$ (Theorem~\ref{prop:qi-rigidity-intro}), building on the work of Peng~\cite{PengCoarseI},~\cite{PengCoarseII} for the part concerning finitely generated groups.

The Dehn function computations are presented in Appendix~\ref{appendix: copmutations}, where we list the groups up to dimension $5$ and compute their image by $\rho_1$ (Tables~\ref{tab:scr_expgrowth-dim4} and~\ref{tab:scr_expgrowth}) as well as their Dehn functions (Tables~\ref{tab:groupsless5prop} and ~\ref{tab:Dehn-functions-5}). Our computations are based on a list of criteria, mostly following~\cite{CoTesDehn}, which are explained in~Section~\ref{sec: Dehn functions}. Finally in Section~\ref{sec: computation examples} we give plenty of examples and explain how to preform the computations.

\subsection{Acknowledgements}
We thank Yves Cornulier for a useful discussion and pointing us to Theorems 6.E.2 and 10.H.1 in \cite{CoTesDehn}. The first author would like to thank Yair Glasner, Ori Parzanchevsky and Shai Evra for their interest and support. 
The second author would like to thank Claudio Llosa Isenrich for a useful discussion.

\section{Completely solvable groups}
\label{sec: completely solvable groups}

A simply connected solvable Lie group $G$ is quasiisometric to a completely solvable group $\rho_0(G)$, called the trigshadow of $G$~\cite{CornulierDimCone}. There are several possible, equivalent definitions for $\rho_0(G)$. Below we give that of Jablonski, building on Gordon and Wilson.

\begin{proposition}[{\cite[\S4.1 and \S4.2]{JabMax}, after \cite{GordonWilson}}]
    Let $G$ be a simply connected solvable Lie group. There exists a (possibly non-unique) left-invariant metric  $g_{\operatorname{max}}$ on $G$ whose isometry group contains a transitive completely solvable group $G_0$. Moreover, the group $G_0$ obtained in this way is unique up to isomorphism and it does not depend on $g_{\max}$.
\end{proposition}

\begin{definition}
    Let $G$ be a simply connected solvable Lie group. We define $\rho_0(G)$ as $G_0$. We say that a group is in the class $(\mathcal C_0)$ if $G = \rho_0(G)$, that is, if $G$ is completely solvable.
    \end{definition}

It is clear that the groups $G$ and $G_0$ are quasiisometric, being closed co-compact subgroups of the isometry group $\widehat G$ of $g_{\operatorname{max}}$. They are commable in the terminology of \cite{CornulierQIHLC}. The role of the group $\widehat G$ is played by the group denoted $G_3$ in Cornulier's treatment (\cite[Lemme 1.3]{Cornulieraspects}, summarizing \cite{CornulierDimCone}).

\begin{definition}
Let $G$ be a group in the class $(\mathcal C_0)$.
The exponential radical $\operatorname{R}_{\exp} G$ of $G$ is the smallest normal subgroup $N$ of $G$ such that $G/N$ is nilpotent.
\end{definition}
The exponential radical was named by Osin \cite{OsinExprad} as it is the subgroup of exponentially distorted elements in $G$ (together with $1$).
We call $\dim G/\operatorname{R}_{\exp} G$ the rank of $G$.
If $\widehat G$ is a real semisimple Lie group with trivial center, writing an Iwasawa decomposition $\widehat G = KAN$ and setting $G=AN$, we recover that the real rank of $\widehat G$ is the rank of $G$. More generally, the rank as defined here is still the dimension of one (or any) Cartan subgroup of $G$.

\subsection{Standard solvable groups}
When the group splits over its exponential radical there are some concrete advantages. In particular, Cornulier and Tessera prove that the following property has many implications for upper bounds on Dehn functions. 

\begin{definition}[After Cornulier and Tessera, {\cite[Definition 1.2]{CoTesDehn}}]
    Let $G$ be a group in $(\mathcal C_0)$, $\mathfrak n = \operatorname{R}_{\exp} \mathfrak g$. We say that $G$ is standard solvable if its exponential radical splits, the quotient $A=G/\operatorname{R}_{\exp} G$ is abelian, and the action of $\mathfrak a$ on $\mathfrak n / [\mathfrak n, \mathfrak n]$ has a trivial kernel.
\end{definition}

The conditions in the definition above are easier to verify than those which appear in~\cite{CoTesDehn}, and it is not entirely obvious why the two definitions are equivalent. Especially, in \cite{CoTesDehn} it is asked that the action of $\mathfrak{a}$ on every proper quotient of $\mathfrak n/[\mathfrak n, \mathfrak n]$ does not admit zero as a nontrivial weight, and $\mathfrak{n}$ is not a-priori required to be the exponential radical.
The following proposition establishes the equivalence of the two definitions (for real Lie groups; the original definition applies in a wider setting including non-Archimedean Lie groups as well), 

\begin{proposition}
\label{prop: UA standard solvable can take u to be ExpRad}
    Let $G$ be a completely solvable group. Assume that $G$ splits as
    $U\rtimes A$, where $A$ is abelian, $U$ is nilpotent, and the action of $A$ on $U/[U,U]$ has no fixed point. Then
    \[ U = [G,G] = \operatorname{R}_{\exp} G. \]
    Moreover, the action of $A$ on any non-trivial quotient of $U/[U,U]$ has no fixed point, hence $G$ is standard solvable in the sense of~\cite{CoTesDehn}.
\end{proposition}

\begin{proof}
    With notation as before, consider 
    \[ \mathfrak g = \mathfrak u \rtimes \mathfrak a, \]
    and denote $\rho: \mathfrak a \to \mathfrak{gl}(\mathfrak u/[\mathfrak u, \mathfrak u])$.
    In order to prove that $\mathfrak{u}\subset [\mathfrak{g},\mathfrak{g}]$, it is enough to prove that $\rho(\mathfrak a)(\mathfrak u/[\mathfrak u, \mathfrak u]) = \mathfrak u/[\mathfrak u, \mathfrak u]$.
    Indeed, if for $u \in \mathfrak  u$, we can write
    \[ u = [a, u'] + w\] where $u' \in \mathfrak u$ and $w \in [\mathfrak u, \mathfrak u] \subseteq [\mathfrak g, \mathfrak g]$, then $u$ is the linear combination of Lie brackets in $\mathfrak g$, so $u\in [\mathfrak g,\mathfrak g]$.
    Let us now prove the claim.
    Let $\mathfrak v = \mathfrak u/[\mathfrak u, \mathfrak u]$ and let $v \in \mathfrak v$. Write 
    $v = v_1 + \cdots + v_s$, where $v_{i} \in V^{\lambda_i}$, $\lambda_i \in \operatorname{Hom}(\mathfrak a, \mathbf R)$ nonzero. It is sufficient to prove that $v_i$ is in the image of $\mathfrak \rho(\mathfrak a)$ for every $i=1,\ldots,s$. Choosing $a_i\in \mathfrak a$ such that $\lambda_i(a_i) \neq 0$, we find that $\rho(a_i)_{\mid V^{\lambda_i}}$ has nonzero diagonal entries, hence it is surjective, so that $v_i \in \rho(a_i)(\mathfrak u/[\mathfrak u,\mathfrak u])$. The fact that $\rho(\mathfrak{a})(\mathfrak u/[\mathfrak u,\mathfrak u])=\mathfrak u/[\mathfrak u,\mathfrak u]$ also implies that in any non-trivial quotients of $\mathfrak u/[\mathfrak u,\mathfrak u]$, the zero weight of the $\mathfrak{a}$-action is trivial. 

    We proved that $\mathfrak u \subseteq [\mathfrak g, \mathfrak g]$. The converse containment follows from the fact that $\mathfrak g/\mathfrak u = \mathfrak a$ is abelian. Now, $C^3 \mathfrak g = [\mathfrak g, \mathfrak g]$ because of the following series of equalities 
    \begin{align*}
        [\mathfrak g, \mathfrak u] = [\mathfrak a + \mathfrak u, \mathfrak u] = [\mathfrak a, \mathfrak u] + [\mathfrak u, \mathfrak u] = \mathfrak u.
    \end{align*}
    and in view of the fact that $\operatorname{R}_{\exp} \mathfrak g$ is the limit of the central series, $\mathfrak u = \operatorname{R}_{\exp} {\mathfrak g}$.
\end{proof}

We can check that all the groups of class $(\mathcal C_0)$ of dimension less than $5$, save for one, are standard solvable (See Table~\ref{tab:groupsless5prop}).
The only exception is the group $G_{4,3}$, since $G_{4,3} /\operatorname{R}_{\exp} G_{4,3}$ is non-abelian. There are many non-standard solvable groups of dimension 5 and class $(\mathcal C_0)$ - see Appendix~\ref{sec: standard solvability}.

\subsection{The class $\mathcal{C}_1$ and sublinear bilipschitz equivalence}\label{Sec: class C1}
It may happen that $G\in \mathcal{C}_0$ does not split over its exponential radical. Cornulier~\cite{CornulierCones11} proved that nonetheless, such $G$ is always coarsely geometrically related to a group $\rho_1(G)$ which does. More precisely, there is always a is sublinear bilipschitz equivalence between $G$ and a group $\rho_1(G)$, where the latter splits over its exponential radical, and moreover the action of $\rho_1(G)/\operatorname{R}_{\exp} \rho_1(G)$ on $\operatorname{R}_{\exp} \rho_1(G)$ is $\mathbf{R}$-diagonalizable. This is the content of the definitions and theorems below.

\begin{definition}
Let $G$ be a completely solvable Lie group with exponential radical $N$.
Say that $G$ is in $(\mathcal C_1)$ if the extension
$1 \to N \to G \to G/N \to 1$ splits and the action of $G/N$ on $N$ is $\mathbf R$-diagonalizable.
\end{definition}

\begin{definition}
    Let $G$ be a completely solvable group with $N = \operatorname{R}_{\exp} G$, and set $H = G/N$.
Decompose $\phi = \operatorname{ad}\colon \mathfrak g \to \operatorname{Der}(\mathfrak n)$ into
\[ \phi = \phi_\delta + \phi_\nu \]
where $\phi_\delta$ is $\mathbf R$-diagonalisable and $\phi_\nu$ is nilpotent \cite{BbkiCartanAlg}.
Note that $\phi_{\delta}$ is zero when restricted to $\mathfrak n$, so that it is well-defined on $\mathfrak h$.
Let $\rho_1(G)$ be $N \rtimes H$, where $\mathfrak h$ acts on $\mathfrak n$ through $\phi_\delta$. We also write $\rho_1(\mathfrak g)$ for $\operatorname{Lie}(\rho_1(G))$.
\end{definition}

Let $X$ and $Y$ be  metric spaces. After fixing $x_0 \in X$ and $y_0 \in Y$, we denote by $\vert \cdot \vert$ the distance to the respective basepoints in $X$ and $Y$. We denote $\vert x_1\vert\vee \vert x_2\vert:=\max\{\vert x_1\vert,\vert x_2\vert\}$. 

Let $u \colon \mathbf R_{\geqslant 0} \to \mathbf R_{\geqslant 1}$ be a sublinear function, that is, 
\[ \lim_{r \to + \infty} \frac{u(r)}{r} =0. \]

\begin{definition}
Let $X$, $Y$, $x_0$, $y_0$ and $u$ be as above. Let $L \geqslant 1$.
    We say that $f\colon X \to Y$ is 
    \begin{itemize}
        \item $(L,u,x_0)$-Lipschitz if for every $x,x' \in X$,
    $d\big(f(x),f(x')\big) \leqslant Ld(x,x') + u (\vert x \vert \vee \vert x' \vert)   $
    \item $(L,u,x_0)$-expansive 
    if $L^{-1}d(x,x') - u (\vert x \vert \vee \vert x' \vert) \leqslant d\big(f(x), f(x')\big)$
    \item $(u,y_0)$-surjective if for every $y$ in $Y$, there is 
    $x \in X$ such that $d\big(y,f(x)\big) \leqslant u(\vert y \vert)$.
    \end{itemize}

\end{definition}

We say that $f$ is a $(L,u)$-bilipschitz embedding if it is $(L,cu,x_0)$-Lipschitz and $(L,cu,x_0)$-expansive for some $c\geqslant 0$.
If $f$ is additionally $(u,y_0)$-surjective for some $y_0$, then for all $y'_0 \in Y$ there is $c'>0$ such that it is $(c'u, y'_0)$-surjective; in this case, we say that $f$ realizes a $\big(L,O(u)\big)$-bilipschitz equivalence, or for short, a $O(u)$-bilipschitz equivalence between $X$ and $Y$. 
When no reference is made to $L$ and $u$ we will call a $(L,u)$-bilipschitz embedding a sublinear bilipschitz embedding.

\begin{theorem}[{Cornulier, \cite{CornulierCones11}}]
\label{th:Cornulier-thm}
Let $G$ be a completely solvable group, and let $H = G/\operatorname{R}_{\exp} G$.
Then
\begin{enumerate}[{\rm (1)}]
\item $G$ and $\rho_1(G)$ are $O(\log)$-bilipschitz equivalent.
    \item $H$ is a $O(\log)$-Lipschitz retract of $G$, more precisely: \label{item:cornulier-retract}
    \begin{enumerate}
        \item $\pi: G \to H$ is $O(\log)$-Lipschitz;
        \item \label{item:cornulier-retract-precise}
        Let $X$ be a nonzero vector in a Cartan subalgebra of $\mathfrak g$. Then there exists $f: H \to G$ (depending on $X$) which is $O(\log)$-Lipschitz and such that 
\begin{enumerate}
\setcounter{enumi}{2}
    \item 
    $\pi \circ f$ is $O(\log)$-close to the identity of $H$.
    \item \label{item:cornulier-retract-implicit}
$f \circ \pi (\exp (tX) ) = \exp(tX)$ for all $t$ in $\mathbf R$.
\end{enumerate}
    \end{enumerate}
\end{enumerate}

\begin{proof}
    Part (1) is stated by Cornulier \cite{CornulierCones11}.
    Parts (2a) and (2bi) express that $\pi\colon G \to H$ is a retract in the category of $O(\log)$-Lipschitz maps, which is also stated in \cite[Example 2.6]{cornulier2017sublinear} and the content of the proof can be found \cite[Theorem 4.4]{CornulierCones11}, where in the notation of \cite{CornulierCones11}, $f$ is the map $\psi^{-1}_{\mid V}$ (before the statement of Theorem 4.4). See also the few lines before  \cite[Lemma 5.2]{CornulierDimCone} where the map $f$ that we need is named $\psi$.
    To check part (2bii) we have to specify the construction of $f$; for this we refer to some parts of Cornulier's proof in \cite{CornulierCones11}. In Cornulier's construction, $\pi(X)$ is identified with an element $\xi$ in a subspace of $V$, the complement of the Lie algebra of $\mathfrak h\cap \operatorname{R}_{\exp} \mathfrak g$ in $\mathfrak h$ where $\mathfrak h$ can be taken to be any Cartan subalgebra of $\mathfrak g$, and then, $f \circ \pi(\exp(tX)) = \exp_G(t\xi)$. Since in our assumption, $X$ lies in a Cartan subalgebra $\mathfrak h$ of $X$, we can take the Cartan subalgebra in the construction of $f$ to be $\mathfrak h$, and then, with this choice, take $\xi$ to be equal to $X$. In this way, $f \circ \pi(\exp(tX))=\exp(tX)$ for all $t \in \mathbf R$.
\end{proof}
\end{theorem}

\section{Distortion in completely solvable groups and Dehn function estimates}
\label{sec:distortion-dehn}

The theory of Dehn function and filling pairs is well known for finitely generated groups, and generalizes naturally for compactly presented groups. We refer to \cite{CoTesDehn} for the basic definitions and a detailed exposition of this subject in the context of compactly presented groups. 

Given two functions $f,g$ from $\mathbf R_{\geqslant 0}$ or $\mathbf Z_{\geqslant 0}$ to itself, we will write $f(r) \preccurlyeq g(r)$ if there is some constant $C>0$ such that $f(r) \preccurlyeq Cg(Cr+C)+Cr+C$. We also write $f(r) \asymp g(r)$ if $f(r) \preccurlyeq g(r)$ and $g(r) \preccurlyeq f(r)$.

\subsection{Proof of Proposition \ref{thm:distort}}
\label{sec:proof-D}

We now proceed with the proof of Proposition \ref{thm:distort}.
We first recall the setting.
    Let $G$ be a completely solvable group, let $H = G/\operatorname{R}_{\exp} G$ and let $\pi: G \to H$ be the projection. Let $X \in \mathfrak g \setminus \{0 \}$ be an element in a Cartan subalgebra, $L$ be the one-parameter subgroup generated by $X$, and let $c_X$ be its central depth.

Our goal is to evaluate the distortion of the subgroup generated by $X$.

    In the case where $c_X =\infty$, $X$ is in $\operatorname{R}_{\exp} G$, and the conclusion follows directly from \cite{OsinExprad}.

    In the case $c_X \neq \infty$, $\pi(L)$ is polynomially distorted with degree $c_X$ in $H$ by \cite{Osindistort}.
    This means that there exists a constant $M \geqslant 1$ so that for $t$ large enough,
    \begin{equation*}
        \frac{1}{M}t^{1/c_X} \leqslant d_H(\exp(t\pi(X)), 1) \leqslant Mt^{1/c_X}
    \end{equation*}
    (here we still denote $\pi$ the map $\operatorname{Lie}(\pi)$ for convenience).
    Now, by Cornulier's theorem \ref{th:Cornulier-thm},  the map $\pi\colon G \to H$  is a $O(\log)$-retract.
    This implies, on the one hand, that $\pi$ is $O(\log)$-Lipschitz. So, for some $\lambda >0$ and $c \geqslant 0$,
    $d_H(\exp(t\pi(X),1) \leqslant \lambda d_G(\exp(tX),1) +c \log t$.
    So $ d_G(\exp(tX),1) \geqslant \frac{1}{\lambda M} t^{1/c_X} - \frac{c}{\lambda} \log t,  $
    and then for $t$ large enough,
    \begin{equation}
        d_G(\exp(tX),1) \geqslant \frac{t^{1/c_X}}{2\lambda M}
    \end{equation}
    Then, for $r>0$ large enough,
    \begin{equation*}
        \sup \{t: \exp tX \in B_G(r) \} \leqslant  (2\lambda M)^{c_X} r^{c_X}.
    \end{equation*}
Thus $\Delta_L^G(r) \preccurlyeq r^{c_X}$.
On the other hand, using part (2b) in Theorem~\ref{th:Cornulier-thm}, there exists a $O(\log)$-lipschitz map $f\colon H \to G$ such that $f(\exp(t\pi(X))) = \exp(tX)$. Taking larger constants $\lambda$ and $c$ if needed, we have that
\begin{align*}
d_G(\exp(tX),1) & \leqslant \lambda d_H(\exp(t\pi(X)),1) + c\log d_H(\exp(t\pi(X)),1) \\
& \leqslant \lambda M t^{1/c_X} + \frac{c}{c_X} \log t,
\end{align*}
so that $\lambda M \Delta_L^G(r)^{1/c_X} + \frac{c}{c_X} \log \Delta_L^G(r) \geqslant r$, and then, $\Delta_L^G(r) \succcurlyeq r^{c_X}$.

\begin{remark}
    For our use of Theorem~\ref{thm:distort}, namely in Proposition~\ref{prop:dehn-lower-lie} below, the element $X$ will always lie in the centre of $\mathfrak{g}$, therefore will always lie in a Cartan subalgebra. In particular the assumption that $X$ lies in a Cartan subalgebra of $\mathfrak g$ does not impose any restrictions to us later on. We do not know whether this assumption is necessary in the statement of Theorem~\ref{thm:distort}.
\end{remark}

    \subsection{From the distortion in an extension to the Dehn function}\label{sec: distortion in Extension to Dehn}

\begin{proposition}
\label{prop:dehn-lower-lie}
    Let $G$ be a simply connected solvable Lie group.
    Let $\omega \in Z^2(G,\mathbf R)$; assume that in the central extension 
    \[ 1 \to \mathbf R \overset{\iota}{\longrightarrow} \widetilde G \overset{\pi}{\to} G \to 1 \]
    associated to $\omega$, the subgroup $L=\iota(\mathbf R)$ is distorted, and $\Delta_L^{\widetilde G}(n) \succcurlyeq n^k$.
    Then the Dehn function of $G$ has growth type at least $n \mapsto n^k$.
\end{proposition}

The main ingredient is the following lemma.

\begin{lemma}
\label{prop:dehn-lower-lie-d-bounded}
    Let $G$ be a simply connected solvable Lie group. Equip $G$ with a left-invariant Riemannian metric.
    Let $\alpha \in \Omega^1(G,\mathbf R)$ be a smooth one-form; assume that
    \begin{itemize}
        \item $d\alpha$ is left-invariant, and
        \item 
        There exists a family $(\gamma_n)_{n \geqslant 1}$ of piecewise smooth loops, with $\operatorname{length}(\gamma_n) \leqslant cn$ and $\int_{\gamma_n} \alpha \geqslant c'n^k$ for some positive constants $c,c'$.
    \end{itemize}
    Then the Dehn function of $G$ has growth type at least $n \mapsto n^k$.
\end{lemma}

\begin{proof}
    We will prove that the filling area $\operatorname{Fill}(\gamma_n)$ of $\gamma_n$ is larger or equal than a constant times $n^k$. For every $n \geqslant 1$ let $\Delta_n$ be a Lipschitz disk in $G$ such that $\partial \Delta_n = \gamma_n$, by which we mean that $\Delta\colon B^2 \to G$ is a Lipschitz embedding of the Euclidean $2$-disk $B^2$ into $G$ such that $\Delta_{n \mid S^1}$ is a reparametrization of $\gamma_n$. $\Delta$ defines an integral Lipschitz chain in $G$, as defined in \cite[2.11]{Federer1974RealFC}. On the other hand, since it is smooth and has bounded exterior derivative, $\alpha$ represents a flat cochain as defined in \cite[4.6]{Federer1974RealFC}. By Federer's version of Stokes' theorem \cite[6.2]{Federer1974RealFC}, 
    \begin{equation}
    \label{eq:stokes}
        \int_{\Delta_n} d\alpha = \int_{\gamma_n} \alpha \geqslant c'n^k .
    \end{equation}
    Now $d\alpha$ is left-invariant, so there is a constant $L>0$ (namely, the point-wise comass norm of $d\alpha$ with respect to the Riemannian metric) such that
    \begin{equation}
    \label{eq:bounding-integral-by-area}
    \left\vert \int_{\Delta_n} d\alpha \right\vert \leqslant L \operatorname{Area}(\Delta_n) =: L \int_{B^2} \vert \Lambda^2 d\Delta (x) \vert dx   
    \end{equation}
    for all $n$. Combining \eqref{eq:stokes} and \eqref{eq:bounding-integral-by-area} yields 
    \begin{equation}
        \operatorname{Area}(\Delta_n) \geqslant \frac{c'}{L} n^k
    \end{equation}
    for all $n$.
    Since $\operatorname{length}(\gamma_n) \leqslant cn$, this finishes the proof that the filling area of $G$, defined by
    \[ \operatorname{Fill}_G(r) = \sup_{\gamma\colon S^1 \to G,\ \operatorname{length}(\gamma) \leqslant r} \inf \{ \operatorname{Area}(\Delta) \colon  \partial \Delta = \gamma  \} \]
    is at least of growth type $n \mapsto n^k$.
    Now $\delta_G(n) \succeq \operatorname{Fill}(n)$ by \cite[Proposition 2.C.1]{CoTesDehn}.\footnote{Actually \cite{CoTesDehn} gives the stronger result that $\delta_G(n) \asymp \operatorname{Fill}_G(n)$, however, the converse inequality is much more involved.}
\end{proof}

\begin{proof}[Proof of Proposition~\ref{prop:dehn-lower-lie} using Lemma~\ref{prop:dehn-lower-lie-d-bounded}]
    Let $G$ and $\omega \in Z^2(G,\mathbf R)$ be as in the statement of Proposition~\ref{prop:dehn-lower-lie}, and let $\alpha$ be a one-form on $G$ such that $d\alpha = \omega$. 
    For all $n$, let $\widetilde \gamma_n$ be a piecewise $C^1$ loop in $\widetilde G$ from $1$ to $\iota(n^k)$. By the assumption on the distortion, we can assume that there is a constant $c>0$ such that $\operatorname{length}(\widetilde \gamma_n) \leqslant cn$. We now let $\gamma_n$ be the projection of $\widetilde \gamma_n$ in $G$. This is a piecewise $C^1$-loop, with length $\leqslant cn$. We now claim that $\int_{\gamma_n} \alpha = n^k$. This follows from \cite[Lemma 3.1]{GMLIP}; there, the Lemma is stated for simply connected nilpotent Lie groups, but the nilpotency assumption is actually not used; the lemma holds for simply connected Lie groups.
\end{proof}

\begin{remark}
    Exponentially distorted central extensions also yield lower bounds on the Dehn function. For a group $G$ of class $(\mathcal C_1)$, the existence of an exponentially distorted central extension is equivalent to the $2$-homological obstruction of Cornulier and Tessera that we will discuss in the next section; see \cite[11.E]{CoTesDehn} on this equivalence.
\end{remark}

\begin{remark}
     Proposition~\ref{prop:dehn-lower-lie} is very close to \cite[Proposition 3.7]{GMLIP}.
     It is slightly stronger, even when restricted to nilpotent groups, since in \cite[Proposition 3.7]{GMLIP} there is the additional assumption that the nilpotent group $G$ should be of nilpotency class $k-1$. The difference comes from the different versions of Stokes theorem used. When $G$ is nilpotent and under the additional assumption that it has a lattice $\Gamma$, the result of Proposition \ref{prop:dehn-lower-lie} can be obtained by combinatorial arguments considering central extensions of $\Gamma$ instead of $G$; this method does not require any assumption on the nilpotency class of $G$ and $\Gamma$. It is described already in \cite{BW97}.
\end{remark}

\subsection{Example}\label{sec: example of central extension}

There are two completely solvable groups of dimension 4 and cone dimension 3. These are $G_{4,3}$, which we define below, and $A_{2} \times \mathbf R^2$. We will prove that the Dehn function of $G_{4,3}$ is cubic, while the Dehn function of $A_2 \times \mathbf R^2$ is quadratic. 

The Lie algebra of $G_{4,3}$ has a basis $(e_1, e_2, e_3, e_4)$ in which the nonzero Lie brackets are
$
[e_4,e_1]  = e_1$ and $
[e_4,e_3]  = e_2$.
(Our $e_4$ is the opposite of the corresponding notation in \cite{PateraZassenhaus}, for convenience; the others are the same.)
The derived subalgebra is the abelian ideal generated by $e_1$ and $e_2$. The next term in the central series (and exponential radical) is $\mathbf R e_1$, and the center is $\mathbf R e_2$.
Consider the dual basis $(\omega_1, \ldots, \omega_4)$ to $(e_1, \ldots e_4)$. The $2$-form $\omega_4 \wedge \omega_2$ is closed, since 
\[ d (\omega_4 \wedge \omega_2) = d\omega_4 \wedge \omega_2 - \omega_4 \wedge d\omega_2 = - \omega_4 \wedge \omega_3 \wedge \omega_4 = 0.\]
It is not exact, since $Z^2(\mathfrak g_{4,3}, \mathbf R)$ is spanned by $\omega_1 \wedge \omega_4$ and $\omega_3 \wedge \omega_4$.
Hence, there is a nontrivial central extension 
\[ 1 \to \mathbf R \to G \to G_{4,3} \to 1. \]
where $G$ is a $5$-dimensional, completely solvable group (this is $G_{5,10}$ on Table~\ref{tab:scr_expgrowth}). Moreover, the kernel of this extension is  cubically distorted, as we explain now. Let us write $\widetilde e_i$ such that $\widetilde e_5$ generates the kernel of the central extension, and $\widetilde e_i$ projects to $e_i$ in $\mathfrak g_{4,3}$ for $1\leqslant i \leqslant 4$. $\widetilde e_5$ lies in the third term of the descending central series of $\mathfrak g$, so that we can apply Proposition \ref{thm:distort}.

\subsection{Comparison with other known bounds on the Dehn functions}\label{sec: comparing lower bounds}

Theorem~\ref{Introthm: Dehn bounds via rho_1} is essentially proven by Cornulier and Tessera. Since they do not state it in this way, we will provide some explanations on  how to deduce it from \cite{CoTesDehn} using their tools. The main ingredient is \cite[Theorem 10.H.1]{CoTesDehn}.
The theorem essentially states that for a Lie group $G$, either the Dehn function is exponential, or it is well estimated (with error terms) by the Dehn function of the largest nilpotent quotient of a completely solvable group quasiisometric to $G$.

We present the ingredients and then assemble the proof. 

\subsubsection{Dehn functions and $O(\log)$-bilipschitz equivalence}
The Dehn functions of two groups that are $O(log)$-bilipschitz equivalent are equal up to a factor of a power of $\log$. We will use it for the pair $G$ and $\rho_1(G)$. The following statements are essentially~\cite[Corollary 3.C.2]{CoTesDehn}, formulated in a slightly more general way. We omit the proof, which is identical. 

\begin{lemma}\label{lem: Dehn bounds in terms of SBE-constants}
    Let $G$ and $H$ be two locally compact compactly presented groups, with filling pairs $(f_G,g_G)$ and $(f_H,g_H)$ respectively. Assume there is an $(L,u)$-bilipschitz equivalence $\phi:H\rightarrow G$. Then we have the following:
    \begin{enumerate}

   \item $f_H(n)\preccurlyeq f_G\big(nu(n)\big)\cdot f_H\circ u\circ g_G\big(nu(n)\big)$.
   
    \item  $g_H(n)\preccurlyeq g_H\circ u\circ g_G\big(nu(n)\big)+g_G\big(nu(n)\big)$.
    \end{enumerate}
\end{lemma}

\begin{corollary}\label{cor: logSBE and poly Dehn implies subexponential dehn}\label{cor: logSBE implies log distortion in Dehn functions}
    If $(f_G,g_G)=(n^d,n^e)$, then $\frac{f_H(n)}{f_H\big(u(n^{2e})\big)}\leq n^{2d}$. If moreover $u=\log$, then  $f_H$ is polynomially bounded. In particular, for two connected Lie groups $G$ and $H$ that are $(L,\log)$-bilipschitz, if $(f_G,g_G)=(n^d,n^d)$ then $f_H\preccurlyeq n^d\log^{2d}(n)$.
    \end{corollary}

\subsubsection{Lower bounds and Lipschitz Retracts}\label{sec: lower bounds and retracts}

Among finitely presented groups, going to a group-theoretic retract decreases the Dehn function, as can be seen by choosing an adequate pair of presentations for which the retract corresponds to an enlargement of the set of relators; see e.g. \cite[Lemma 1]{BaumslagMillerShort}.
This is still valid for Lie groups; this fact is used several times in \cite{CoTesDehn} but we provide a proof 
for completeness.

\begin{proposition}\label{prop: lower bound on retracts}
    Let $G$ be a simply connected Lie group, and let $H$ be a retract of $G$ in the Lie group category. Then $\delta_H \preccurlyeq \delta_G$. 
\end{proposition}

\begin{proof}
    Consider the epimorphism $\pi : G \to H$ and let $\sigma \colon H \to G$ be a section of $\pi$. Let $d_G$ be a left-invariant Riemannian distance on $G$, and let $d_H$ be a left-invariant Riemannian distance on $H$ such that $\pi$ is a Riemannian submersion. 
    Then for every $h,h' \in H$,
    \begin{equation}
    \label{eq:submersion-cosets-distance-formula}
        d_H(h,h') = \operatorname{dist}_G(\sigma(h)N, \sigma(h')N)
    \end{equation}
    where $N= \ker \pi$ (see e.g. \cite[Lemma 4.6]{HigesPeng}).
    Let $\gamma: S^1 \to H$ be a Lipschitz loop of $d_H$-length exactly $n$, and consider the loop $\widehat \gamma = \sigma \circ \gamma$.  The $d_G$-length of $\widehat \gamma$ is less or equal to $n$ thanks to \eqref{eq:submersion-cosets-distance-formula}; it is also greater or equal than $n$, since $\pi$ is $1$-Lipschitz and sends $\widehat \gamma$ onto $\gamma$.  Using the equivalence of the Dehn function and the filling function in $G$ \cite[Proposition 2.C]{CoTesDehn}, there is a filling of $\widehat \gamma$ in $G$ by a Lipschitz disk $\Delta$ of area at most the order of $\delta_G(n)$. Since $\pi\colon (G,d_G) \to (H,d_H)$ is $1$-Lipschitz, $\pi \circ \Delta$ has area less than $\Delta$. Using again the equivalence of the Dehn function and the filling function, in $H$ and in the reverse direction, we conclude that $\delta_H(n) \preccurlyeq \delta_G(n)$.
\end{proof}

\subsubsection{Generalized Standard Solvable Groups}

A special case of interest where the group $G$ retracts to a subgroup is when the short exact sequence determined by the exponential radical of $G$ splits. This case is captured by Cornulier and Tessera's definition of \emph{generalized standard solvable} groups: 

\begin{definition}[\cite{CoTesDehn}, Section~10.H.1] \label{def: gen stnd solvabl}
    Let $G$ be a completely solvable group. We call $G$ \emph{generalized standard solvable} if $G=V\rtimes N$ where $N$ is nilpotent and such that the following condition on the action of $N$ on $V$ is met: there is no nontrivial quotient of $V/[V,V]$ on which $N$ acts as the identity.
\end{definition}

If $G\in\mathcal{C}_0$ admits a splitting $G=V\rtimes N$ as a generalized standard solvable group, then if $N$ is abelian then $V=\operatorname{R_{exp}}(G)$ (Proposition~\ref{prop: UA standard solvable can take u to be ExpRad}) and $G$ is standard solvable. In general, any such splitting with $N$ nilpotent forces $V$ to contain the exponential radical. On the other hand, if $G\in \mathcal{C}_0$ is moreover in $(\mathcal{C}_1)$, it is  automatically generalized standard solvable with $V=\operatorname{R_{exp}}(G)$:

\begin{lemma}\label{lemma: C_1 automatic generalized standard solvable}
    If $G$ is of class $(\mathcal C_1)$, then it is generalized standard solvable via the splitting $G=\operatorname{R}_{\exp}G\rtimes N$.
\end{lemma}

\begin{proof}
    Let $U$ be the exponential radical of $G$ and assume towards contradiction that $H_1(\mathfrak u)$ has zero as a nontrivial weight. Let $X$ be a nonzero vector in the corresponding kernel, and let $\widehat X \in \mathfrak u$ be such that $X = \widehat X + [\mathfrak u, \mathfrak u]$. Then $[\mathfrak n,\widehat X] \subseteq [\mathfrak u, \mathfrak u]$. 
    But since $\mathfrak u$ is the exponential radical of $\mathfrak g$, one has $[\mathfrak g,\mathfrak u] = \mathfrak u$, especially $[\mathfrak g,\mathfrak u]$ should contain $\widehat X$. However the map 
    $\mathfrak g \times \mathfrak u \to \mathfrak u$ which to $(Y,U)$ associates $[Y,U]$ is not surjective, since its image cannot contain $\widehat X$ (remember that since $G$ is in $(\mathcal C_1)$, the action of $\mathfrak n$ on $\mathfrak u/[\mathfrak u, \mathfrak u]$ is diagonalizable. So a nonzero vector in the kernel cannot be in the image). This is a contradiction.
\end{proof}

The following is the main result of Cornulier and Tessera on generalized standard solvable groups. See Section~\ref{sec: computation examples} for the definition of $\operatorname{Kill}(\mathfrak{v})$.

\begin{theorem}[\cite{CoTesDehn}, Theorem 10.H.1]\label{thm: 10.H.1} Let $G=V\rtimes N$ be a generalized standard solvable group whose Dehn function is non-exponential (i.e.\ strictly smaller than exponential).  Then $\delta_G(n)\preccurlyeq n \cdot \widehat{\delta_N}(n)$, where $\widehat{\delta_N}$ denotes any regular function larger than $\delta_N$. If, moreover, $\operatorname{Kill}(\mathfrak{v})_0=0$ then $\delta_G(n)\preccurlyeq \widehat{\delta_N}(n)$.  
\end{theorem}

We are now ready to complete the proof of Theorem~\ref{thm: Dehn bounds via rho_1} assuming the results we took from \cite{CoTesDehn}.

\begin{proof}[Proof of Theorem~\ref{thm: Dehn bounds via rho_1}]
Let $G$ be a Lie group. It is quasiisometric to a completely solvable group $G_0$ and $\delta_G\asymp \delta_{G_0}$. In turn, $G_0$ is $O(\log)$ bilipschitz equivalent to $G_1:=\rho_1(G_0)$ (Theorem~\ref{th:Cornulier-thm}), hence Corollary~\ref{cor: logSBE implies log distortion in Dehn functions} gives $\delta_{G_1}(n)/\log ^e(n)\preccurlyeq \delta_{G_0}(n)\preccurlyeq \delta_{G_1}(n)\cdot \log ^e(n)$. By definition, $G=\operatorname{R_{exp}}G\rtimes N$ hence $\delta_N(n)\preccurlyeq\delta_{G_1}(n)$ by Proposition~\ref{prop: lower bound on retracts}. By Lemma~\ref{lemma: C_1 automatic generalized standard solvable} $G_1$ is generalized standard solvable, and Theorem~\ref{thm: 10.H.1} gives $\delta_{G_1}(n)\preccurlyeq n\cdot \widehat{\delta_N}(n)$. Combining all inequalities completes the proof. 
\end{proof}

A careful read of the proof of Theorem~\ref{thm: Dehn bounds via rho_1} sheds light on the theoretical contribution of Theorem~\ref{thm:distort} and Proposition~\ref{prop:dehn-lower-lie} over their well known nilpotent groups analogues. If $G$ is in $(\mathcal{C}_1)$ and $N=G/\operatorname{R_{exp}}G$, then any lower bound on $\delta_N$ (in particular those coming from distorted central extensions) is automatically a lower bound on $\delta_G$. However in general in order to retract to a  nilpotent group one might have to pass to $\rho_1(G)$, which comes at a cost of a power of $\log$ factor on the lower bound. Our version allows using distorted central extensions without passing to the nilpotent quotient, therefore removing  this factor.

The distortion arising from central extensions as in Corollary~\ref{corIntro: distorion central extensions} cannot be used to distinguish $\delta_G$ from $\delta_N$, because every polynomially distorted central extension of $G$ is a pull back of a central extension of $N$. More precisely:

\begin{lemma}\label{lem: distortion in central extension to nilpotent quotient}
    Let $G\in \mathcal{C}_0$, $N=G/\operatorname{R_{exp}}G$.
    Let $\omega \in Z^2(G,\mathbf R)$; assume that in the central extension 
    \[ 1 \to \mathbf R \overset{\iota}{\longrightarrow} \widetilde G \overset{\pi}{\to} G \to 1 \]
    associated to $\omega$, the subgroup $L=\iota(\mathbf R)$ is distorted, and $\Delta_L^{\widetilde G}(n) \asymp n^k$. Then $\omega$ is the pull back of some $\eta \in Z^2(N,\mathbf R)$ such that in the associated central extension 
     \[ 1 \to \mathbf R \overset{j}{\longrightarrow} \widetilde N \overset{\pi}{\to} N \to 1 \] 
     the subgroup $M=j(\mathbf{R})$ is distorted with $ n^k \preccurlyeq \Delta_M^{\widetilde N}(n)$.
\end{lemma}

\begin{remark}
By~\cite[Theorem 11.C.1]{CoTesDehn}), the hypothesis that $L$ is polynomially distorted in $\tilde{G}$ can be replaced by `$G$ has polynomially bounded Dehn function'.
\end{remark}

\begin{corollary}
    In the setting of Lemma~\ref{lem: distortion in central extension to nilpotent quotient}, $\delta_N(n)\succcurlyeq n^k$.
\end{corollary}

\begin{proof}
    By Theorem~\ref{thm:distort},  the hypothesis implies $L < C^k G$, and it is enough to prove that that there exists such $\eta$ with $M < C^k N$, the $k$\textsuperscript{th} term of the lower central series.  
    Denote $U:=\operatorname{R}_{\exp}G$, so $N=G/U$. Further denote  $\tilde{U}:=\operatorname{R}_{\exp}\tilde{G}$. The proof is straightforward, using the definitions of brackets in central extensions and functoriality of passing to exterior algebras. The map $f:\mathfrak g\rightarrow \mathfrak n= \mathfrak g/\operatorname{R_{\exp}}\mathfrak g$, induces  $f^*:\Lambda^\ast(\mathfrak n^\ast)\rightarrow \Lambda^\ast(\mathfrak g^\ast)$. This induces a map between the Lie group cohomology of $N$ and that of $G$, which we may still call $f^*:H^2(N,\mathbf{R})\rightarrow H^2(G,\mathbf{R})$. 

    Our aim is to show that the image of $f^*$ contains $\omega$. Let $\{e_1,e_2,\dots e_n\}$ be a basis for  $\mathfrak{g}=\operatorname{Lie}(G)$, and $F\in \tilde{\mathfrak{g}}$ the central element with $L=\{\exp(tF):t\in\mathbf{R}\}$. Let $\{\omega_1,\omega_2\dots,\omega_n\}$ be the basis for $\Lambda^1(\mathfrak{g}^\ast)$ dual to $\{e_1,e_2,\dots e_n\}$, and for  $1\leq i<j\leq n$, $\omega_{i,j}=\omega_i\wedge\omega_j$, the standard basis for $\Lambda^2(\mathfrak{g}^\ast)$. Write $\omega=\sum \alpha_{i,j}\omega_{i,j}$. 
    
    As a central extension, recall that the defining brackets in $\tilde{\mathfrak{g}} \simeq \mathfrak g \oplus \mathbf RF$ are  
    $$[(e_i,\alpha)(e_j,\beta)]_{\tilde{\mathfrak{g}}}=([e_i,e_j]_{\mathfrak{g}},\alpha_{i,j})$$
    From this it is clear that whenever $\alpha_{i,j}\ne 0$ it holds that $e_i,e_j\notin \mathfrak{u}$: since the distortion of $\exp((X,0))\in \tilde{G}$ is at least as large as the distortion of $\exp(X)\in G$, we have $e_i\in \mathfrak{u}\Rightarrow (e_i,0),([e_i,e_j],0)\in \tilde{\mathfrak{u}}$. Therefore if $\alpha_{i,j}\ne 0$ then 
    $$e_i\in \mathfrak{u}\Rightarrow [(e_i,0),(e_j,0)]_{\tilde{\mathfrak{g}}}-([e_i,e_j]_\mathfrak{g},0)=\alpha_{i,j}F\in\tilde{\mathfrak{u}}$$

    Since we assume $k<\infty$, we may conclude $e_i,e_j\notin\mathfrak{u}$ whenever $\alpha_{i,j}\ne 0$. For $X\in \mathfrak{g}$, denote $\bar{X}:=f(X)\in \mathfrak{n}$ its image under the projection. We may assume the set $\{\bar{E}_1\bar{E}_2,\dots,\bar{E}_m\}$ is a basis for $\mathfrak{n}$, and $\bar{\omega}_i$ the dual basis for $\mathfrak{n}^\ast$. The above claim says that $\alpha_{i,j}\ne 0\Rightarrow i,j\leq m$, and we may consider $\bar{\omega}=\sum_{\alpha_{i,j}\ne 0}\alpha_{i,j}\bar{\omega}_{i,j}\in\mathfrak{n}\wedge \mathfrak{n}$. Functoriality implies that $\bar{\omega}$ is a non-trivial cohomology class. Consider the central extension $\tilde{\mathfrak{n}}$ generated by $\bar{\omega}$, and let $\bar{F}$ be the corresponding central element in $\tilde{\mathfrak{n}}$. It remains to show that $\bar{F}\in \bar{\mathfrak{n}}^k$. This is a result of the fact that we can define the Lie algebra homomorphism $\tilde{f}:\tilde{\mathfrak{g}}\rightarrow \tilde{\mathfrak{n}}$ by $(e_i,0)_{\tilde{\mathfrak{g}}}\mapsto (\bar{E}_i,0)_{\tilde{\mathfrak{n}}}$ and $F=(0,1)_{\tilde{\mathfrak{g}}}\mapsto \bar{F}=(0,1)_{\tilde{\mathfrak{n}}}$.
\end{proof}

\subsection{An application: a lower bound on the Dehn function of a central product of Abels' second group}\label{app: Example}\label{example: contribution of central extensions theorem}
In this section we elaborate on Example~\ref{ExampleIntro: Abels group}, in which we present a group for which our theoretical contribution is practical: applying Corollary~\ref{corIntro: distorion central extensions} to this group improves on the known lower bound obtained by Theorem~\ref{thm: Dehn bounds via rho_1}. The group is a central product of Abels' second group with a model filiform nilpotent group of class 4. Remark~\ref{rmk: logic of example} provides some intuition for its construction, which is a variation  on~\cite[Examples 4.1, 4.2]{CornulierDimCone}, using~\cite[Example 1.5.4]{CoTesDehn} as a building block.

\begin{remark}
    The example we construct is of dimension 13, and we do not claim the dimension is minimal for the properties we need. Upon establishing the list of Dehn functions for completely solvable groups of dimensions 4 and 5, we found that in the few cases where the distortion in central extensions tool was useful, the groups were in fact in $(\mathcal{C}_1)$. So the minimal dimension for such an example is bounded below by 6. 
\end{remark}

Let $\mathfrak{g}_2$ be the Lie algebra corresponding to the group $U\rtimes A$ presented in~~\cite[Example 1.5.4]{CoTesDehn}. The Lie algebra $\mathfrak{u}=\operatorname{Lie}(U)$ is given by: 
\begin{align*}
    \mathfrak{u}:=&\langle X_1,X_2,X_3,X_4,X_5,X_6,X_9,X_{12}\rangle
    \end{align*}
    with the nonzero brackets
    \begin{align*}
&[X_1,X_2]=X_4,[X_1,X_3]=X_5, [X_2,X_3]=X_6,\\
&[X_1,X_6]=X_9, 
[X_3,X_4]=X_{12}, [X_2,X_5]=X_9+X_{12}.
\end{align*}
(The indices of the generators indeed skip 7, 8, 10 and 11: this is reminiscent to the fact that $\mathfrak{u}$ is the quotient of the free $3$-step nilpotent Lie algebra on $3$ generators $\{X_1,X_2,X_3\}$ by the ideal generated by $[X_i,[X_i,X_j]]$ for $i\ne j\in \{1,2,3\}$).

Let $\mathfrak{a}:=\operatorname{Lie}(A)$ be the abelian Lie algebra on two generators $\langle T_1,T_2\rangle$, and $\mathfrak{g}_2:=\mathfrak{u}\rtimes \mathfrak{a}$ with the following non-zero  brackets: 
\begin{align*}   
  &[T_1,X_1]=-X_1,[T_1,X_3]=X_3,[T_1,X_4]=-X_4,[T_1,X_6]=X_6,\\
  &[T_2,X_1]=-X_1,[T_2,X_2]=2X_2,[T_2,X_3]=-X_3,[T_2,X_4]=X_4,[T_2,X_5]=-2X_5,\\ &[T_2,X_6]=X_6
\end{align*}

Let $\mathfrak{fil} =\langle E_1,E_2,E_3,E_4\rangle$ be the 4-dimensional filiform algebra, with non-zero brackets $[E_1,E_2]=E_3,[E_1,E_3]=E_4$. Define $\mathfrak{g}_3:=\mathfrak{g}_2\times \mathfrak{fil}$. The centre of $\mathfrak{g}_3$ is the product of the centres of the factors, which is $\langle X_9,X_{12},E_4\rangle$. Let $\mathfrak{z}$ be the $1$-dimensional subspace generated by the diagonal of the centre $Z:=X_9+X_{12}+E_4$. 

Our example is the simply connected Lie group $G$ whose Lie algebra is $\mathfrak{g}:=\mathfrak{g}_3/\mathfrak{z}$. We can write it in the basis 
$$\langle X_1,X_2,X_3,X_4,X_5,X_6,X_9,X_{12},T_1,T_2,E_1,E_2,E_3\rangle,$$

with all non-zero brackets exactly as in $\mathfrak{g}_3$, except for $[E_1,E_3]=-X_9-X_{12}$.

The group $G$ admits the following properties: 

\begin{itemize}
    \item $\rho_0(G)=G$, i.e. $G$ is in $ (\mathcal{C}_0)$.
    \item  $U=\operatorname{R}_{\exp} G$ and the short exact sequence 
    \[ 1 \to U \to G \to G/U \to 1 \]
    does not split. In particular $G$ is not in $(\mathcal C_1)$.
    \item $G$ is not generalized standard solvable.
    \item $G$ does not have a nonabelian nilpotent retract. 
    \item $G$ has a polynomially bounded Dehn function. 
    \item $G$ admits a cubically distorted central extension. 
\end{itemize}

We supply short reasoning for the above claims. Due to the high dimension of this group, we do not give the details of the computations. The reader may consult Section~\ref{sec: computation examples} for the relevant techniques.

The nilradical of $\mathfrak{g}$ is $\mathfrak{u}+\mathfrak{fil}/\langle Z\rangle$, it splits with  complement $\mathfrak{a}$ acting with only real eigenvalues. Therefore $G$ is in $(\mathcal{C}_0)$. 

The  exponential radical is $\mathfrak{u}$. The quotient $\mathfrak{g}/\mathfrak{u}$ is $\mathbf{R} ^2\times \mathfrak{heis}$, where in the above basis $\mathfrak{heis}=\langle E_1,E_2,E_3\rangle/\mathfrak{u}$ is the Lie algebra of the $3$-dimensional real Heisenberg group with central element $E_3\cdot \mathfrak{u}$. If $\mathfrak{u}$ did split, we would have $\mathfrak{g}=\mathfrak{u}\rtimes \mathfrak{m}$ with $\mathfrak{m}$ isomorphic to $\mathbf{R}^2\times \mathfrak{heis}$ and its action on $\mathfrak{u}$ would factor through the quotient. In particular, the central element of $\mathfrak{heis}$ in $\mathfrak{m}$, which is central in $\mathfrak{m}$, would act trivially on $\mathfrak{u}$. Therefore the centre of $\mathfrak{g}$ would intersect $\mathfrak{m}$ nontrivially. It can be verified however that the centre of $\mathfrak{g}$ is exactly $\langle X_9,X_{12}\rangle\subset \mathfrak{u}$.  

If, towards contradiction, $G$ was generalized standard solvable with nilpotent quotient $N$, then $\mathfrak{n}:=\operatorname{Lie}(N)$ would have to be a quotient of $\mathfrak{m}=\mathbf{R}^2\times \mathfrak{heis}$ (recall $G/\operatorname{R_{exp}(G)}$ is the largest nilpotent quotient of $G$). Therefore if $\mathfrak{n}$ is nonabelian, it must contain $\mathfrak{heis}$ and the same argument as above yields a contradiction. This moreover proves that $G$ does not retract to a nonabelian nilpotent group.  
If on the other hand $N$ was abelian, then by Proposition~\ref{prop: UA standard solvable can take u to be ExpRad} $G$ would have to split over the exponential radical, which is not the case. So $G$ is not generalized standard solvable.  

One may check that $G$ does not admit the SOL or $2$-homological obstructions, and therefore has a polynomial Dehn function~\cite[Theorem E]{CoTesDehn}. 

It is easily observed that $G_3$ (the Lie group corresponding to $\mathfrak{g}_3$) is a central extension of $G$. The generator of the extension is $Z=X_9+X_{12}+E_4$, which is in $C^{3}\mathfrak{g}_3$ but not in $C^4\mathfrak{g}_3$, i.e. $c_{Z}=3$. By Theorem~\ref{thm:distort}, $\langle Z\rangle $ is $n^3$-distorted in $G_1:=\rho_1(G)$, hence by Corollary~\ref{corIntro: distorion central extensions}, the Dehn function of $G$ is bounded from below by $n^3$. A direct computation of the second cohomology group proves that we cannot improve this lower bound using distortion in other central extensions.

To the best of our knowledge, the sharpest evaluation of the Dehn function of $G$ prior to our work is given by Theorem~\ref{thm: Dehn bounds via rho_1}. Concretely, it is easy to check that $\rho_1(\mathfrak{g})=\mathfrak{u}\rtimes (\mathfrak{a}\times \mathfrak{heis})$, where the only difference in brackets from $\mathfrak{g}$ is that $[E_1,E_3]=0$. In particular, the action of $\mathfrak{a}\times \mathfrak{heis}$ on the exponential radical $\mathfrak{u}$ is the same as in $\mathfrak{g}$ and so the corresponding Lie group $G_1=U\rtimes (A\times \operatorname{Heis})$ is generalized standard solvable in the sense of~\cite[Section 10.H]{CoTesDehn}. Theorem~\ref{thm: Dehn bounds via rho_1} yields: 
\[ n^3/\log^e(n) \preceq \delta_G(n) \preceq  n^4\cdot \log^e(n). \]

where $e\leq 8$ is twice the bound on the exponent of $\delta_G$.

Finally, we remark that $\operatorname{Kill}(\mathfrak{u})_0\ne 0$, so the upper bound of $n^4$ cannot be improved to $n^3$ using only~\ref{thm: 10.H.1}. Moreover, $\rho_1(G)$ is not generalized tame as can be seen by drawing the weight diagram, and using~\cite[Proposition 4.B.5]{CoTesDehn}.

\begin{remark} \label{rmk: logic of example}
    The logic behind the construction is extracted from the discussion at the end of Section~\ref{sec: comparing lower bounds}. We are looking for a group with quite a few properties: it must be completely solvable, must not split with a nilpotent quotient (in particular it should not be generalized standard solvable, hence not in  $(\mathcal{C}_1)$), its Dehn function must be polynomial, and its distorted central extensions must be necessary for giving a lower bound on its Dehn function (so for example it must not admit a left-invariant nonpositively curved Riemannian metric). In low dimensions and for the groups of class $(\mathcal{C}_0)$, the property of not splitting over the exponential radical is the hardest to come by. Cornulier's construction~\cite[Examples 4.1, 4.2]{CornulierDimCone} give such groups, and hints at how to obtain them in general. We vary his building blocks in order to assure that the group admits the other desired properties; the challenge of escaping the SOL obstruction is the main reason we chose~\cite[Example 1.5.4]{CoTesDehn} as a building block for our example. 

\end{remark}
 
\section{The solvable Lie groups of low dimensions and their QI-invariants}\label{sec: computations and contribution}

In Appendix~\ref{appendix: copmutations} we list all indecomposible completely solvable groups of exponential growth of dimensions 4 and 5, and compute their cone dimension  and Dehn functions. The results are organized in tables: cone dimension in Tables~\ref{tab:scr_expgrowth-dim4},~\ref{tab:scr_expgrowth}, Dehn function in Tables~\ref{tab:groupsless5prop} and~\ref{tab:Dehn-functions-5}.

We use the results of our computations in the proof of the QI-rigidity of $\operatorname{Sol}_5$ (Theorem~\ref{prop:qi-rigidity-intro}, see proof in Section~\ref{sec:qi-rigid}). In a different paper~\cite[Section 5]{productSBE} we use them in a slightly different manner, namely in order to find groups that share the same cone dimension and Dehn function. As two prominent quasiisometric invariants, one may expect such a list could find more applications in the context of the quasiisometric classification, at least in low dimensions.

We do not introduce new methods: our computations are based mostly on Cornulier and Tessera~\cite{CoTesDehn}, who develop various criteria for estimating Dehn functions. Their work is remarkable in its completeness: for example, it allows us to determine all Dehn functions within the class of groups we considered. 

Still, the computations themselves are technically demanding and require familiarity with~\cite{CoTesDehn} as well as with other works. With the aim of making this paper as self-contained as possible, along with the results we supply in the appendix short descriptions and examples for each criterion. These should allow the readers to familiarize themselves with the relevant definitions and techniques, and to reproduce the computations.

\begin{remark}
    The Dehn functions for all solvable Lie groups of dimension up to $5$ and polynomial growth follow from those of nilpotent groups of these dimensions, which were completely determined by Pittet \cite[Proposition 7.1]{PITTET_1997}. We restrict our computations to the groups of exponential growth in this paper.
\end{remark}

\subsection{Some particular families}\label{sec: particular families}

We discuss in detail two particular families, to indicate where some progress would be needed to complete their quasiisometry classification. We will say that a given completely solvable Lie group $G$ is QI-rigid within $(\mathcal C_0)$ if any group in $(\mathcal C_0)$ quasiisometric to it is isomorphic to $G$, and that a class $\mathcal{G}$ of groups is QI-complete within $(\mathcal C_0)$ if any group in $(\mathcal C_0)$ quasiisometric to a group in $\mathcal G$ is isomorphic to a group in $\mathcal G$. 
All notions and notations in the discussion below could be found in Appendix~\ref{appendix: copmutations}. 
\subsubsection{$G_{4,8}$ and the $G_{4,9}^\beta$ family}

\begin{figure}
    
\begin{tikzpicture}[line cap=round,line join=round,>=angle 45,x=1.2cm,y=1.2cm]
\clip (-4.5,-1.1) rectangle (5,1);

\draw[->,color=black] (-4.5,0) -- (4.5,0);

\draw[line width= 2pt] (-4,0) -- (0,0);

\draw (4.5,0) node[right]{\scriptsize $\beta$};

\fill (0,0) circle(2pt) node[anchor=south east] {\scriptsize $0$};
\fill (-4,0) circle(2pt) node[anchor=south east] {\scriptsize $-1$};
\fill (-2,0) circle(2pt) node[anchor=south east] {\scriptsize $-\frac{1}{2}$} node[anchor=north]{\scriptsize QI to $\mathbf R^2 \rtimes \operatorname{SL}(2,\mathbf R)$ } ;;
\fill (4,0) circle(2pt) node[anchor=south west] {\scriptsize $1$} node[anchor=north]{\scriptsize QI to $\operatorname{SU}(2,1)$} ;
\draw (4,-0.5) node[anchor=center]{\scriptsize QI-rigid};
\path[draw,decorate,decoration=brace] (0,0.2) -- (4,0.2)
node[midway,above]{\footnotesize QI-classified for $0<\beta<1$ \cite{CarrascoSequeira}};

\draw (-2,-1) node[stuff_fill]{\scriptsize $\delta_G\asymp \exp$ (SOL obstruction holds)};
\draw (2,-1) node[stuff_fill]{\scriptsize $\delta_G$ linear};

\end{tikzpicture}
    \caption{Together with $G_{4,8}$, the groups in the $G_{4,9}^{\beta}$ family are determined by the parameter $\beta \in (-1,1]$. The three classes of groups defined by $\beta >0$, $\beta=0$ and $\beta<0$ are quasiisometrically distinct. The groups with $\beta <0$ (with exponential Dehn function) have not been classified up to quasiisometry so far.}
    \label{fig:the-g4-9-family}
\end{figure}
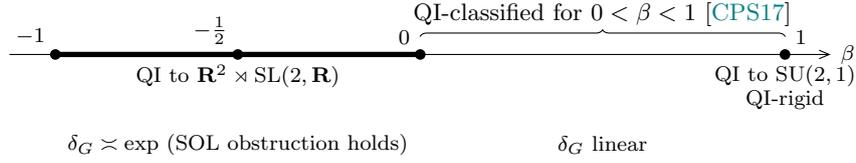

Together with $G_{4,8}$, the groups of the form $G_{4,9}^\beta$ may be represented on a line segment, so that the eigenvalues of $e_4$ acting on the abelianization of $\operatorname{R}_{\exp} \mathfrak g_{4,9}$ are $1$ and $\beta \in (-1,1]$ (see Figure~\ref{fig:the-g4-9-family}.) Note that $G_{4,8}$ is the limit case $\beta=-1$.
The group $G_{4,9}^{1}$ is the maximal completely solvable subgroup of $\operatorname{SU}(2,1)$; as such, it is QI-rigid within $(\mathcal C_0)$; See~\cite[Appendix A]{productSBE}. When $\beta >0$ the group $G_{4,9}^\beta$ is hyperbolic; the fact that the $\{G_{4,9}^\beta \colon \beta >0 \}$ is QI-complete within $(\mathcal C_0)$ can be deduced from \cite{KLDNG}, and the internal QI-classification is done by \cite{CarrascoSequeira}.
The group $G_{4,9}^0$ is not in $(\mathcal C_1)$ and its cone dimension is $2$. This is the only group in the family with this cone dimension.
When $\beta<0$ the cone dimension is again $1$ and the SOL obstruction holds, so that the Dehn function is exponential.

The group $G_{4,9}^{-1/2}$ is of particular interest, since it is quasiisometric to $G= \operatorname{SL}(2,\mathbf R) \ltimes \mathbf R^2$; precisely it is isomorphic to its subgroup $ANR$, where $R$ denotes the radical $\mathbf R^2$ and  $L=KAN$ denotes the Levi factor $\operatorname{SL}(2,\mathbf R)$ in the Levi decomposition $LR$ of $G$. De la Harpe \cite[IV.25.(viii)]{delaHarpeTopics} observed that the lattices in $G$ are nonuniform and asked whether there are finitely generated groups quasiisometric to $G$.

\begin{ques}
    Can one describe the space $\operatorname{QI}(G_{4,9}^\beta, G_{4,9}^{\beta'})$ for $\beta, \beta'<0$?
\end{ques}

In the special case $\beta=\beta'=-1/2$, this amounts to the knowledge of the group of self-quasiisometries of $G_{4,9}^{-1/2}$ and would likely shed light on de la Harpe's question mentionned above.

\subsubsection{The $G_{5,33}^{\alpha, \beta}$ family.}

\begin{figure}
    
\begin{tikzpicture}[line cap=round,line join=round,>=angle 45,x=1.2cm,y=1.2cm]
\clip (-4.5,-3) rectangle (4.5,3.5);

\draw[->,color=black] (-4.5,0) -- (2.5,0);
\draw[->,color=black] (0,-2.5) -- (0,2.5);

\draw[line width= 2pt] (-4.5,0) -- (0,0);
\draw[line width= 2pt] (0,-2.5) -- (0,0);

\draw (0,2.5) node[above]{\scriptsize $\beta$};
\draw (2.5,0) node[right]{\scriptsize $\alpha$};

\draw [ dash pattern = on 2pt off 2pt] (-1,2) -- (2,-1);

\draw (-2pt,2pt) -- (2pt,-2pt);
\draw (-2pt,-2pt) -- (2pt,2pt);

\fill (1,0) circle(2pt) node[anchor=north east] {\scriptsize $(1,0)$};

\draw (-2,0) node[stuff_fill]{\scriptsize $\delta_G\asymp \exp$};
\fill (-1,-1) circle(2pt) node[anchor=east]{\tiny $(-1,-1)$} ;
\draw (-1,-1.1) node[stuff_fill, anchor=north east] {\scriptsize (See Section \ref{sec:qi-rigid}) };

\draw (2.2,2.2) node[ stuff_fill]{\footnotesize QI-classified \cite{bourdon2023rhamlpcohomologyhigherrank}};
\draw [->] (2,2) -- (-0.5,1.5);
\draw [->] (2,2) -- (1.5,-0.5);

\fill (1,0) circle(2pt) node[anchor=north east] {\scriptsize $(1,0)$};
\fill (0,1) circle(2pt) node[anchor=south west] {\scriptsize $(0,1)$};

\draw (-2,0.5) node[stuff_fill]{\scriptsize $\delta_G(n) \asymp n^2$};

\draw (-2,-0.5) node{\scriptsize $\delta_G(n) \asymp n^2$};

\draw (2,-1.6) node[stuff_fill]{\scriptsize AW condition holds;};
\draw (2,-2) node[stuff_fill]{\scriptsize $\operatorname{Cone}_\omega G_{5,33}^{\alpha, \beta}$ contractible};

\draw (-2,-2) node{\scriptsize $\pi_2(\operatorname{Cone}_\omega G_{5,33}^{\alpha, \beta}) \neq 0$ \cite{Cornulieraspects}};

\end{tikzpicture}
    \caption{The groups in the $G_{5,33}^{\alpha, \beta}$ family are determined by the coordinates $(\alpha,\beta)$ of the third principal weight in the above weight diagram, so that one point in the plane represents a group; $G_{5,33}^{\alpha, \beta}$ and $G_{5,33}^{\alpha',\beta'}$ are isomorphic if and only if $\{\alpha, \beta\} = \{\alpha', \beta'\}$.  The group $G_{5,33}^{\alpha, \beta}$ with $\alpha = 0$ is decomposable.
    The groups in the three areas: the bottom-left quadrant, its boundary (bold-faced), and their complement (to the top and right) are quasiisometrically distinct.  $G_{5,33}^{-1,-1}$ is the group $\operatorname{Sol}_5$ discussed in Section~\ref{sec:qi-rigid}.}
    \label{fig:the-g5-33-family} 
\end{figure}
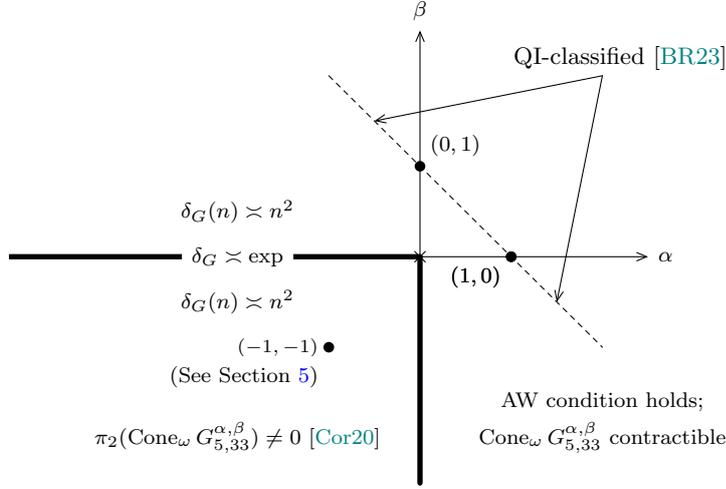

Groups of the form $G_{5,33}^{\alpha, \beta}$ may be represented on a plane, so that the three principal weights in the basis of $\operatorname{Hom}(\mathfrak a, \mathbf R)$ dual to $({e_4},{e_5})$ are $(1,0)$, $(0,1)$, and $(\alpha, \beta)$. See Figure \ref{fig:the-g5-33-family}.
This family is interesting because it exhibit various behaviours.

If $\alpha$ or $\beta$ is strictly positive, then $G_{5,33}^{\alpha, \beta}$ has the Azencott-Wilson criterion; as such, it has a quadratic Dehn function.
If moreover $\alpha + \beta =1$, then $G_{5,33}^{\alpha, \beta}$ contains $G_{4,5}^{1,1}$ (namely the $AN$ subgroup of $KAN=\operatorname{SO}(4,1)$) as a codimension $1$ subgroup. 
Bourdon and Rémy recently used this fact to compute critical exponents in $L^p$-cohomology for the groups within the line $\alpha+\beta=1$; they obtain that two such groups are quasiisometric if and only if they are isomorphic \cite[Theorem 1]{bourdon2023rhamlpcohomologyhigherrank}.

If $\alpha$ and $\beta$ are nonpositive, on the other hand, then $0$ lies in the convex hull of the set of principal weights, which changes drastically the geometry. If $\alpha$ or $\beta$ is zero, then $G_{5,33}^{\alpha, \beta}$ has the SOL obstruction (as $0$ lies in the segment between two principal weights), therefore its Dehn function is exponential. If $\alpha$ and $\beta$ are both negative, then the Dehn function is again quadratic, as we compute in Example~\ref{exm:computing-killing-g-5-33}. However, we can still distinguish the groups $G_{5,33}^{\alpha, \beta}$ with $\alpha, \beta<0$ from the other ones using the asymptotic cone: $\pi_2(\operatorname{Cone}_\omega G_{5,33}^{\alpha, \beta})$ is nontrivial for $\alpha, \beta<0$ (See \cite[p.9]{deCornulier2014aspects}; the asymptotic cone is $\mathbf T^3_{D(3)}$) while when $\alpha$ or $\beta$ are positive the asymptotic cone is contractible (since it is bilipschitz to a CAT(0) space). The group $G_{5,33}^{-1,-1}$ is the only unimodular group in the $G_{5,33}$ family; it is in Peng's class $(\mathcal P)$ (See Definition~\ref{def:peng-class}) and the description of its self-quasiisometries is given by \cite{PengCoarseI,PengCoarseII}. See Section~\ref{sec:qi-rigid} for more on Peng's class and the group $G_{5,33}^{-1,-1}$.

\section{Groups quasiisometric to \texorpdfstring{$\operatorname{Sol}_5:$ }: proof of Theorem~\ref{prop:qi-rigidity-intro}}
\label{sec:qi-rigid}

Eskin, Fisher and Whyte have conjectured that the class of virtually polycyclic groups should be quasiisometrically complete within finitely generated groups \cite[Conjecture 1.2]{EskinFisher}.
The first evidence for this came from the work of Shalom, who proved that any finitely generated group quasiisometric to a polycyclic group has a finite index subgroup with nonvanishing first Betti number \cite{ShalomHarmonic}.
Since every polycyclic group contains a finite-index subgroup which is a uniform lattice in a simply connected solvable Lie group, the quasiisometry classification conjecture for completely solvable Lie groups \cite[Conjecture 19.25]{CornulierQIHLC} would complement \cite[Conjecture 1.2]{EskinFisher} in that it would complete the internal QI-classification of polycyclic group. (Note that when we restrict attention to the smaller class of virtually nilpotent groups, we face a similar picture, but while the quasiisometric classification of simply connected nilpotent Lie groups is still open, Gromov's polynomial growth theorem \cite{GromovPolyG} can be taken as a replacement of the above conjecture of Eskin, Fisher and Whyte.)

Peng \cite{PengCoarseI, PengCoarseII} made significant progress towards both conjectures. To state her theorems we make the following definition. 

\begin{definition}
\label{def:peng-class}
    Let $G$ be a standard solvable group. We say that $G$ is of class $(\mathcal P)$ if the following holds:
    \begin{enumerate}
        \item $\mathrm{R}_{\exp}G$ is equal to the nilradical of $G$
        \item $\mathrm{R}_{\exp}G$ is abelian
        \item $G/\mathrm{R}_{\exp} G$ is abelian
        \item $G$ is unimodular.
    \end{enumerate}
\end{definition}

\begin{theorem}[{\cite[Corollary 5.3.7]{PengCoarseII}}]\label{thm:peng-classification}
    If two groups of class $(\mathcal P)$  are quasiisometric, then they are isomorphic.
\end{theorem}

\begin{theorem}[{\cite[Corollary 5.3.9]{PengCoarseII}}]\label{thm:peng-rigidity}
    If a finitely generated group is quasiisometric to a group of class $(\mathcal P)$, then it is virtually polycyclic.
\end{theorem}

\begin{remark}
The statement of Theorem~\ref{thm:peng-classification} in~\cite{PengCoarseII} is different and involves real parts of Jordan form of the adjoint action. This is because Peng's definition of class  $(\mathcal{P})$ is more general: the groups are not assumed to be standard solvable, and in particular need not be in $(\mathcal{C}_0)$ a priori.  With the current formulation it is easy to see that \cite[Corollary 5.3.7]{PengCoarseII} implies \cite[Corollary 5.3.8]{PengCoarseII}.
\end{remark}

The completely solvable groups of class $(\mathcal P)$ and dimension less or equal $5$ are $G_{3,4}$, $G_{4,2}^{-2}$, $G_{4,5}^{-(1+\delta)/2, -(1-\delta)/2}$ for $0 \leqslant \delta < 1$,
$G_{5,7}^{\alpha, \beta, \gamma}$ with $\alpha+\beta+\gamma=1$, $G_{5,9}^{-1-\delta, -1+\delta}$ with $0\leqslant \delta <1$, $G_{5,11}^{-3}$, $G_{5,15}^{-1}$, and $G_{5,33}^{-1,-1}$, which is the group $\mathrm{Sol}_5$. 

Using Peng's rigidity theorem and our Dehn functions computations (Appendix~\ref{appendix: copmutations}), we can prove that the finitely generated groups quasiisometric to $\mathrm{Sol}_5$ are almost lattices in this group.

\begin{proposition}
\label{prop:qi-rigidity}
    Let $\Gamma$ be a finitely generated group quasiisometric to $\mathrm{Sol}_5$. Then there is a finite-index subgroup $\Gamma_0$ in $\Gamma$ and a homomorphism $\Gamma_0 \to \mathrm{Sol}_5$ with finite kernel and closed co-compact image.
\end{proposition}
\begin{remark}
    The group $\mathrm{Sol}_5$ does have lattices - see Remark~\ref{Rem: finite-index-subgroup} below.
\end{remark}

\begin{remark}\label{Rem: finite-index-subgroup}
    Passing to a finite index subgroup is necessary, as the following example shows. Let $K$ be a number field of degree $3$ with Galois group $\Sigma = \operatorname{Sym}_3$ and consider the group $\Lambda = \operatorname{PSL}(2,\mathcal O_K)$, which embeds as a non-uniform $\mathbf Q$-rank one lattice in $X = \mathbb H^2 \times \mathbb H^2 \times \mathbb H^2$. The group $\Sigma$ operates by automorphisms on $\Lambda$; these extend as conjugations in $G=\operatorname{Isom}(X)$ by maps of $\mathbb H^2 \times \mathbb H^2 \times \mathbb H^2$ permuting the three factors.
    If $\mathcal{H}$ is a horosphere bounding a cusp in the quotient $\Gamma \backslash X$, then by \cite[Proposition 2.1(3)]{PrasadStrongRigidity}, $\Lambda \rtimes \Sigma$ intersects $\operatorname{Isom}(\mathcal H)$ in a uniform lattice of the latter group; let us denote this lattice by $\Gamma$. One may write $\Gamma = \Gamma_0 \rtimes \Sigma$, where $\Gamma_0$ is a lattice in $\operatorname{Sol}_5 < \operatorname{Isom}(\mathcal H)$. Thus $\Gamma$ is quasiisometric to $\operatorname{Sol}_5$. However, if there was $\phi\colon \Gamma \to \operatorname{Sol}_5$ a group homomorphism with finite kernel, then
    $\ker \phi$ would contain $\Sigma$ (which is torsion), but intersect $\Gamma_0$ trivially, since $\Gamma_0<\operatorname{Sol}_5$ has no non-trivial finite subgroup; it would follow that $\ker \phi = \Sigma$ and $\Gamma$ would split as a direct product, which is not the case. 
\end{remark}

\begin{proof}
The proof is done by exhaustion.
    The group $\mathrm{Sol}_5$ is of class $(\mathcal P)$, so that by Peng's rigidity theorem~\ref{thm:peng-rigidity} we know that $\Gamma$ is virtually polycyclic; let $G$ be a simply connected solvable group such that there exists a finite-index subgroup of $\Gamma$ that surjects with finite kernel onto a uniform lattice in $G$  \cite[Theorem 4.28]{RagDS}. 
    We know that $G$ is quasiisometric to $\mathrm{Sol}_5$, so that $\dim G=5$ and $\operatorname{conedim} G =2$. 
    Therefore $G$ is among the following list of groups:
    \[ G_{5,19}^{0,\beta},\ G_{5,19}^{1,\beta},\ G_{5,20}^0,\ G_{5,20}^1,\ G_{5,27},\ G_{5,28}^1,\ G_{5,30}^1,\ G_{5,32}^0,\ G_{5,32}^{\alpha},\ G_{5,33}^{\alpha,\beta}, G_{5,34},\ G_{5,35},  \]
    \[ G_{5,36},\ G_{5,37},\ G_{4,2}^{\alpha} \times \mathbf R, G_{4,4} \times \mathbf R, G_{4,5}^{\alpha, \beta} \times \mathbf R,\ G_{4,7} \times \mathbf R,\ G_{4,8} \times \mathbf R, G_{4,9}^0 \times \mathbf R,\ G_{4,9}^\beta \times \mathbf R. \]
    We can rule out most of the groups in this list since they are not unimodular. The unimodular ones are
    \[ G_{5,19}^{1,-2},\ G_{5,20}^0,\ G_{5,33}^{-1,-1},\ G_{5,35}^{0,-2},\ G_{4,2}^{-2} \times \mathbf R,\ G_{4,5}^{\alpha, \beta \colon \alpha+\beta=-1} \times \mathbf R,\ G_{4,8} \times \mathbf R. \]
    By our work in Tables \ref{tab:groupsless5prop}-\ref{tab:Dehn-functions-5} all these except $G_{5,33}^{-1,-1} \simeq \mathrm{Sol}_5$ have exponential Dehn function; the Dehn function of $G$ is quadratic, since it is quasiisometric to $\mathrm{Sol}_5$ (by \cite[Theorem 1.1]{DrutuFillingSol}, Leuzinger-Pittet \cite[Theorem 2.1]{LeuzingerPittetQuadratic}, or the computation in Example~\ref{exm:computing-killing-g-5-33}, using Cornulier and Tessera's \cite[Theorem F]{CoTesDehn}). So $G$ must be isomorphic to $\mathrm{Sol}_5$.
\end{proof}

\begin{remark}
    According to Peng \cite[Corollary 5.3.11]{PengCoarseII}, a completely solvable group quasiisometric to $\mathrm{Sol}_5$ must be a semidirect product of the form $\mathbf R^2 \ltimes \mathbf R^3$. This would allow to rule out some of the groups above without estimating Dehn functions, but it does not rule out $G_{4,5}^{\alpha, 1-\alpha} \times \mathbf R$. 
\end{remark}

\begin{proposition}\label{prop:qi-rididity-C0}
    Let $G$ be a completely solvable group, quasiisometric to $\mathrm{Sol}_5$. Then $G$ is isomorphic as a Lie group to $\mathrm{Sol}_5$. 
\end{proposition}

\begin{proof}
The group $\mathrm{Sol}_5$ is amenable and unimodular, hence geometrically amenable, and geometric amenability is invariant under quasiisometry, so that $G$ must be geometrically amenable (See \S 11 in \cite{TesseraLSSobolev} and especially Corollary 11.13 therein). Moreover, since $G$ is completely solvable, it is geometrically amenable if and only if it is unimodular. So $G$ is unimodular as well, and belongs to the list of groups already considered in the proof of Proposition~\ref{prop:qi-rigidity}. The end of the proof is the same as that of Proposition~\ref{prop:qi-rigidity}.
\end{proof}

\appendix 
\renewcommand{\thesection}{\Alph{section}}
\renewcommand{\thesubsection}{\Alph{section}.\arabic{subsection}}

\section{Computing Dehn Functions of Completely Solvable Groups up to Dimension 5}  \label{appendix: copmutations}
\label{sec: Presenting format of tables and Dehn function criteria}

We begin by presenting the format of the tables and some basic terminology and context. The tools for Dehn function computations are presented in Section~\ref{sec: Dehn functions}, and elaborated examples with detailed explanations on the computations are given in Section~\ref{sec: computation examples}.

By dimension we mean the dimension of the Lie algebra over $\mathbf R$; for the simply connected solvable Lie groups this is also the asymptotic Assouad-Nagata dimension (\cite{HigesPeng}) so that, for instance, the family of completely solvable groups of dimension $5$ is QI-complete within $(\mathcal C_0)$.

By cone dimension, we mean the covering dimension of any asymptotic cone; it is given by Cornulier's formula \cite{CornulierDimCone}, and for simply connected solvable Lie groups, it is exactly the codimension of the exponential radical.
The cone dimension is obviously a quasiisometry invariant, so that the first refinement of the simply connected solvable groups in QI-complete families is the ordered pair of positive integers
\[ (\mathrm{conedim},\, \mathrm{dim}). \]
The cone dimensions of Lie groups up to dimension 4 was computed and tabulated by Kivioja, Le Donne and Nicolussi Golo \cite[Table 1]{KLDNG}. The same authors also listed the simply connected solvable Lie groups $G$ of polynomial growth of dimension $5$, and their associated $\rho_0(G)$ (that they call the real shadow of $G$).

The solvable Lie groups (or more precisely the real solvable Lie algebras) of dimensions $4$ and $5$ were completely classified by Mubarakzyanov \cite{Mubarakzyanov}; the list is also available in the more accessible \cite{PateraZassenhaus}.  We list in Table \ref{tab:scr_expgrowth-dim4} the groups $G$ in $(\mathcal C_0)$ and of dimension $2$ to $4$ and of exponential growth that do not split in direct product and their associated $\rho_1(G)$ in $(\mathcal C_1)$ (the cone dimensions can be found in \cite{KLDNG}).
In Tables \ref{tab:scr_expgrowth}--{\ref{scr_expgrowth_ctd}} we continue the list in dimension 5 to all indecomposable simply connected solvable groups of exponential growth $G$, and compute their associated $\rho_0(G)$; we also compute $\rho_1(G)$ and list the cone dimension.
There are 39 families of indecomposable real five-dimensional solvable Lie algebras, including 18 with parameters.
In \cite{Mubarakzyanov} they are named $g_{5,i}$ for $1\leqslant i \leqslant 39$. For $1\leqslant i \leqslant 7$, $g_{5,i}$ is nilpotent and for $i \in \{14,17,18,26 \}$ and certain particular values of the parameters, the corresponding group has polynomial growth; we deliberately exclude them from our tables, since the computation of $\rho_0$ was done for them in \cite[Table 3]{KLDNG}.

In order to ease the determination of whether a given irreducible simply connected solvable group belongs to $(\mathcal C_0)$ or $(\mathcal C_1)$, we always list the group in the rightmost possible column. For instance for some group $G$ in $(\mathcal C_1)$, the column below $G$ and (in Tables~\ref{tab:scr_expgrowth}--\ref{scr_expgrowth_ctd}) the column $\rho_0(G)$ will be left blank, only the column $\rho_1(G)$ will be filled with $G$.

The structure of the Lie algebra is given in \cite{Mubarakzyanov} and \cite{PateraZassenhaus} as a list of nonzero brackets; however this is not quite convenient when it comes to computing $\rho_1(G)$, and checking our computations.
It turns out that in all cases but two, namely $G_{5,38}$ and $G_{5,39}$, the nilradical is split, and the Lie algebra decomposes as $\mathfrak n \rtimes \mathbf R$ or $\mathfrak n\rtimes \mathbf R^2$, where the Lie algebra $\mathfrak n$ of the nilradical can be either $\mathbf R^d$ for $d \in \{1,\ldots, 4 \}$, the Lie algebra $\mathfrak{heis}$ of the Heisenberg group, the Lie algebra $\mathfrak{fil}$ of the 4-dimensional filiform group, or a product of $\mathfrak{heis}$ with an abelian ideal of dimension $1$. We fix bases $(e_1, \ldots, e_d)$ for all the Lie algebras among the former, in the following way: $(e_1,e_2,e_3)$ is a basis of $\mathfrak{heis}$ in which $[e_1,e_2] = e_3$ and $e_3$ is central, $(e_1,e_2,e_3,e_4)$ is the basis of $\mathfrak{fil}$ in which $[e_1,e_2] = e_3$, $[e_1,e_3] = e_4$ and $e_4$ is central; when we write the product $\mathbf R \times \mathfrak{heis}$ the nonzero bracket is $(e_2, e_3) = e_4$ while when $\mathfrak{heis} \times \mathbf R$ the nonzero bracket is $(e_1, e_2) = e_3$. 
To denote the torus of derivations, we write 
$\Delta(\mathbf b_1, \ldots, \mathbf b_r)$ for a derivation of $\mathfrak n$ which has diagonal blocks $\mathbf b_1$, \ldots ,  $\mathbf b_r$. By block we mean one of the following:
\begin{itemize}
    \item a scalar block corresponding to an eigenspace of eigenvalue $\lambda \in \mathbf R$ which we write $\mathbf b_i = (\lambda)$.
    \item a complex scalar block corresponding to 
    \[ \begin{pmatrix} \sigma & \tau \\ - \tau & \sigma \end{pmatrix} \]
    which we write $(\sigma \pm i\tau)$.
    \item a non-scalar irreducible Jordan block of a generalized eigenspace of eigenvalue $\lambda$, which we write $(\lambda^s)$ where $s$ is the dimension, e.g. $(2^2)$ corresponds to the block  
    \[ \begin{pmatrix} 2 & 1 \\ 0 & 2 \end{pmatrix} \]
\end{itemize}
Finally,
\[ \mathfrak g = \mathfrak n \rtimes \{ \Delta(\mathbf b_1, \ldots ,\mathbf b_{r_i}) \}_i \]
denotes the semidirect product of $\mathfrak n$ by the derivations listed.
For a few real Lie algebras of dimension 5, the description above is not possible and we provide it separately.

The parameters in our table are in the same range\footnote{In a few cases, a relevant range for the parameters is not clearly indicated in \cite{PateraZassenhaus}, but the invariants given there can be used to determine one. We provide explicit ranges here.} and order as in \cite{PateraZassenhaus}, who took the list in \cite{Mubarakzyanov}, but denoted $A_{5,i}^{a,b,c...}$ the corresponding algebras with parameters; however we used Greek letters for the parameters, and sometimes used different letters; this is in order to ``type'' the parameters, for instance $\tau$ always denotes an imaginary part and $\epsilon$ is a sign. We normalize as much as possible to reduce the number of parameters when applicable, but we did not reparametrize.

\begin{remark}
We did not find the Lie algebra named $\mathfrak s_{5,26}$ in \cite[p.237]{WinternitzSnobl}, for the value of the parameter $a=1$, in \cite{Mubarakzyanov} nor in \cite{PateraZassenhaus}. We named the corresponding group $S_{5,26}^1$ in our table; that is our only departure from the taxonomy of Mubarakzyanov.
    In addition, there is an entry in  \cite{PateraZassenhaus}, named $A_{5,40}$ and marked there as solvable, however it turns out that it has a nontrivial Levi decomposition; this is $\mathfrak{sl}(2,\mathbf R) \ltimes \mathbf R^2$ with the tautological representation of $\mathfrak{sl}(2,\mathbf R)$.
    There are two quasiisometry classes of Lie groups with this Lie algebra:
    \begin{enumerate}
        \item The simply connected Lie group $$ \widetilde{\mathrm{SL}(2,\mathbf R)} \ltimes \mathbf R^2$$ whose QI type is that of $\mathbf R \times G_{4,8}$;
        \item The connected Lie group $\operatorname{SL}(2,\mathbf R) \ltimes \mathbf R^2$, whose QI type is that of $G_{4,9}^{-1/2}$; we discuss this further in subsection \ref{sec: particular families}.
    \end{enumerate}
    Since they are not solvable, we do not list these groups. Otherwise, we found a few less serious inconsistencies:
    in \cite{PateraZassenhaus} one should not have a parameter $c$ in the definition of $G_{5,8}$; in $G_{5,9}^{b,c}$ we found the condition $bc \neq 0$ to be missing in both \cite{Mubarakzyanov} and \cite{PateraZassenhaus}; in the definition of $G_{5,26}^{p, \epsilon}$, resp. of $G_{5,33}^{a,b}$ one should assume $p >0$, resp. $a \leqslant b$ to avoid redundancy.
\end{remark}

\begin{remark}\label{example: non-split radical in dim 5}
    The exponential radical of $G_{5,20}^0$ is not split. This group $G_{5,20}^0$ is the group named $G$ in \cite[Example 4.2]{CornulierDimCone}. There are no completely solvable groups for which the exponential radical does not split in dimension four, so that Cornulier's example is minimal for the dimension.
    It follows from our study that $G_{5,20}^0$ is the only such group in dimension $5$, so the six-dimensional $\mathbf Q$-algebraic group with no $\mathbf Q$-split torus that has a non-split exponential radical given in \cite[Example 4.1]{CornulierDimCone} in $(\mathcal C_0)$ is minimal for the dimension among groups with all these properties.
\end{remark}

\subsection{Dehn functions}\label{sec: Dehn functions}

In tables~\ref{tab:groupsless5prop}--\ref{tab:Dehn-functions-5} we compute, as accurately as we can,  the Dehn functions of the groups listed in Tables~\ref{tab:scr_expgrowth-dim4}-\ref{scr_expgrowth_ctd}. To this end we used the following list of criteria: 

\begin{itemize}
    \item Gromov-hyperbolicity \cite{Heintze}; 
    \item Azencott-Wilson criterion~\cite{azencott1976homogeneous};
    \item Standard solvability~\cite{CoTesDehn};
    \item SOL obstruction~\cite{CoTesDehn};
    \item $2$-homological obstruction~\cite{CoTesDehn};
     \item Vanishing of the zero weight subspace in the Killing module~\cite{CoTesDehn};
     \item Bound for generalized tame groups~\cite{CoTesDehn}; 
    \item The distortion of all central extensions as Lie group (Section~\ref{sec: distortion in Extension to Dehn}).

\end{itemize}
To the best of our knowledge, this list exhausts the general, or algorithmic, criteria for estimating Dehn functions. For each group the table states how we obtain this Dehn function, so our computations could be easily verified. We now explain how we obtain our estimates based on the above criteria. 
\begin{itemize}
    \item The Azencott-Wilson criterion checks whether a Lie group acts simply transitively on a non-positively curved Riemannian manifold \cite{azencott1976homogeneous}. In particular such groups have Dehn function at most quadratic. 
    \item The definitions of \emph{standard solvable}, \emph{{\rm SOL} obstruction} and \emph{$2$-homological obstruction} are given by Cornulier and Tessera. They prove the following result for a completely solvable group $G$ (\cite{CoTesDehn}, Theorem E): 
    \begin{itemize}
        \item $G$ has exponential Dehn function if and only if it admits either the SOL obstruction or the $2$-homological obstruction. Otherwise its Dehn function is polynomially bounded. 
        \item If $G$ does not satisfy the SOL or $2$-homological obstruction and is standard solvable, then it has at most cubic Dehn function.  
    \end{itemize}
    The group $G$ is \emph{generalized tame} if it can be written as $G=U\rtimes N$, N nilpotent compactly generated with some element $c\in N$ acting as a compaction on $U$ (see~\cite[Section 6.E]{CoTesDehn}). Theorem~6.E.2 of~\cite{CoTesDehn} states that in this case, the Dehn functions of $G$ and $N$ are almost equivalent:  $\delta_N\preccurlyeq\delta_G \preccurlyeq \widehat{\delta_N}$, where $\widehat{\delta_N}$ is any regular function larger than $\delta_N$ (see Theorem~\ref{Introthm: Dehn bounds via rho_1} for the definition of a regular function). In all cases of completely solvable groups of dimensions 4 and 5, $\widehat{\delta_N}$ can be taken to equal $\delta_N$ so $\delta_G\asymp \delta_N$.
    
    \item For standard solvable groups there are two ways of concluding they have at most quadratic Dehn function. The first is  using Theorem D in~\cite{CoTesDehn} stating that for some strong version of standard solvable groups, not having the SOL obstruction implies a quadratic upper bound on the Dehn function.  The second is using Theorem~10.E.1 in~\cite{CoTesDehn}, which involves computing the zero weight subspace in  the \emph{Killing module}. Finally, if a standard solvable group has a co-dimension 1 exponential radical and the group does not satisfy the SOL obstruction, it is hyperbolic (\cite{CoTesDehn}, Corollary E.3.a).

    \item Section~\ref{sec: distortion in Extension to Dehn} allows to derive Dehn function estimates for a group $G$ using its possible central extensions. The $2$-cohomology of a group classifies its central extensions. We compute it, and for each possible extension check its distortion. In the tables we list the names of the cohomology classes that give the maximal degree distortions. The name of a cohomology class is given with respect to the ordered basis in which the group is presented in tables~\ref{tab:scr_expgrowth-dim4}--\ref{tab:scr_expgrowth}. 
    
\end{itemize}

For readability we only write down the Dehn function and the Reason column, that depicts which criteria were used to concluded the Dehn function. The reason is enough in order to completely determine which criteria led us to this decision. For example in $G_{5,10}$ we write $n^4$ D.Ex $\{\omega_{15},\omega_{23}\}$. This means that the distorted central extensions corresponding to the cohomology classes of $\omega_{15}$ and $\omega_{23}$ are $n^4$ distorted and that there are no higher degree distorted extension; here $\omega_{ij}$ is dual to $e_i \wedge e_j$ in the given basis of the Lie algebra. 
The fact that the group is not standard solvable and does not admit the Azencott-Wilson criterion follows from these facts. 

In the Reason column, we use the following abbreviations for our reasoning: 

\begin{itemize}
    \item SOL: $G$ admits the SOL obstruction.
    \item Hyp: $G$ is standard solvable with co-dimension $1$ exponential radical and does not admit the SOL obstruction. 
    
    \item not Hyp: a sufficient condition to check non-hyperbolicity is when the cone dimension is larger than $1$.
    
    \item $\rho_1$: The Dehn functions of $G$ and $\rho_1(G)$ are related by Cornulier's Theorem~\ref{th:Cornulier-thm} and Proposition~\ref{cor: logSBE implies log distortion in Dehn functions}.
    
    \item $\rho_1=\rho_0$: When $\rho_1=\rho_0$ then $G$ is quasiisometric to $\rho_1(G)$ and all Dehn functions equal. 
    
    \item A-W: $G$ admits Azencott-Wilson criterion, and so its Dehn function is either linear (if and only if $G$ is hyperbolic), or quadratic. 
    
    \item $n^k$ D.Ex.: The group admits a $n^k$-distorted central extension, and does not admit extensions of higher degree distortions.

    \item C-T: A group that is standard solvable and does not admit SOL or $2$-homological obstructions has at most cubic Dehn function. 

    \item K: A standard solvable group without SOL or $2$-homological obstruction can admit the Killing module criterion, by which its Dehn function is at most quadratic.

    \item GT: If $G=U\ltimes N$ is generalized tame, then the Dehn function of $G$ can be estimated from the Dehn function of its largest nilpotent quotien
\end{itemize}

\begin{table}[t]
    \begin{center}
    \begin{tabular}{lllccc}
    \toprule
         $G$ & $\rho_1(G)$ & structure of $\mathfrak g = \operatorname{Lie}(G)$ &  \scriptsize{$\operatorname{conedim}$} \\
         \midrule
         - & $A_2$ & $\mathbf R \rtimes \Delta(1)$ & 1  \\
         $G_{3,2}$ & $G_{3,3}$ & $\mathbf R^2 \rtimes \Delta(1^2)$ & 1\\
         - & $G_{3,3}$ & $\mathbf R^2 \rtimes \Delta(1,1)$ & 1  \\
         - & $G_{3,4}$ & $\mathbf R^2 \rtimes \Delta(1,-1)$ & 1 \\
         - & $G_{3,5}^\alpha$ & $\mathbf R^2 \rtimes \Delta(1,\alpha)$, $-1 < \alpha < 1$, $\alpha \neq 0$. & 1 \\
         $G_{4,2}^\alpha$  & $G_{4,5}^{1, \alpha'}$\textsuperscript{($\dagger$)}  & $\mathbf R^3 \rtimes \Delta(1^2,\alpha)$, $\alpha \neq 0$. & 1 \\
        - &  $G_{4,3}$ &  $\mathbf R^3 \rtimes \Delta(1,0^2)$ & 3 \\
        $G_{4,4}$& $G_{4,5}^{1,1}$ &  $\mathbf R^3 \rtimes \Delta(1^3)$ & 1   \\
        - & $G_{4,5}^{\alpha, \beta}$ &  $\mathbf R^3 \rtimes \Delta(1,\alpha, \beta)$,  & 1   \\
        & &  $-1\leqslant \alpha \leqslant \beta \leqslant 1$, $\alpha\beta \neq 0$. \\
        $G_{4,7}$ & $G_{4,9}^1$ & $\mathfrak{heis} \rtimes \Delta(1^2,2)$ & 1  \\
        - & $G_{4,8}$ & $\mathfrak{heis} \rtimes \Delta(1,-1,0)$ & 1  \\
         $G_{4,9}^0$ & $\mathbf{R} \times G_{3,3}$ & $\mathfrak{heis} \rtimes \Delta(1,0, 1)$.  & 2   \\
        - & $G_{4,9}^\beta$ &   $\mathfrak{heis} \rtimes \Delta(1,\beta, 1+\beta)$, $-1 < \beta \leqslant 1$, $\beta \neq 0$ & 1  \\
         \bottomrule
    \end{tabular}
    \end{center}
    
    \footnotesize{\textsuperscript{($\dagger$)} $\alpha'$ may differ from $\alpha$.}
  
    \caption{The completely solvable, indecomposable Lie groups of exponential growth and dimension $d$, $2 \leqslant d \leqslant 4$.}
    \label{tab:scr_expgrowth-dim4}
\end{table}

\begin{table}[H]
    \centering
    \begin{tabular}{llllc}
    \toprule
         $G$ & $ \rho_0(G) $ & $\rho_1(G)$ & structure of $\mathfrak g = \operatorname{Lie}(G)$ & $\operatorname{conedim}(G)$ \\
         \midrule
         - &  - & $G_{5,7}^{\alpha, \beta, \gamma}$ & $\mathbf R^4\rtimes \operatorname{diag}(1, \alpha, \beta, \gamma)$,  & 1 \\
         & & & $- 1\leqslant \gamma \leqslant \beta \leqslant \alpha \leqslant 1,  \alpha \beta \gamma \neq 0$. & \\
         - & - & $G_{5,8}^\gamma$ & $\mathbf R^4 \rtimes \Delta(0^2,1,\gamma) $, & 3 \\
         & & & $- 1\leqslant \gamma\leqslant 1,  \gamma \neq 0$. & \\
         - &$G_{5,9}^{\beta,\gamma}$ & $G_{5,7}^{1,\beta, \gamma}$ & $\mathbf R^4 \rtimes \Delta(1^2, \beta, \gamma)$, $\beta\gamma \neq 0$, $\beta \leqslant \gamma$. & 1 \\
         - & - & $G_{5,10}$ & $\mathbf R^4 \rtimes \Delta(0^3,1)$ & 4 \\
         - & $G_{5,11}^\gamma$ & $G_{5,7}^{1,1,\gamma}$ & $\mathbf R^4 \rtimes \Delta(1^3,\gamma)$, $\gamma \neq 0$. & 1\\
         - & $G_{5,12}$ & $G_{5,7}^{1,1,1}$ & $\mathbf R^4 \rtimes \Delta(1^4)$ & 1\\
         $G_{5,13}^{\alpha, 0, 1}$ & - & $\mathbf R^2 \times G_{3,5}^\alpha$ & $\mathbf R^4 \rtimes \Delta(1,\alpha, \pm i)$, $\alpha \neq 0$. & 3 \\
         $G_{5,13}^{\alpha, \beta, \tau}$ & - & $G_{5,7}^{ \alpha, \beta, \beta}$ & $\mathbf R^4 \rtimes \Delta(1,\alpha, \beta \pm  i \tau)$,  & 1 \\
         & & & $-1\leqslant \alpha \leqslant 1, \,\alpha \beta \tau \neq 0$. & \\
         $G_{5,14}^\alpha$ & - & $G_{5,8}^1$ & $\mathbf R^4 \rtimes \Delta(0^2,\alpha \pm i)$, $\alpha \neq 0$. & 3 \\
         - & $G_{5,15}^0$ & $G_{5,8}^1$ & $\mathbf R^4 \rtimes \Delta(0^2,1^2)$   & 3 \\
         - & $G_{5,15}^\beta$ & $G_{5,7}^{1,\beta, \beta}$ & $\mathbf R^4 \rtimes \Delta(1^2,\beta^2)$, $\beta \neq 0$.& 1 \\
         $G_{5,16}^{0, \tau}$ & - & $\mathbf R^2 \times G_{3,3}$ & $\mathbf R^4 \rtimes \Delta(\pm i\tau,1^2)$, $\tau \neq 0$. & 3 \\
         $G_{5,16}^{\beta, 1}$ & - & $G_{5,7}^{1,\beta, \beta}$ & $\mathbf R^4 \rtimes \Delta(1^2,\beta \pm i)$, $\beta \neq 0$. & 1 \\
         $G_{5,17}^{\tau, 0,1}$ & - & $\mathbf R^2 \times G_{3,3}$ & $\mathbf R^4 \rtimes \Delta(\pm i, 1 \pm i \tau)$, $\tau \neq 0$. &  3 \\
         $G_{5,17}^{\tau, \alpha, \beta}$ & - & $G_{5,7}^{1, \beta/\alpha, \beta/\alpha}$ & $\mathbf R^4 \rtimes (\alpha \pm i, \beta \pm i\tau) $ & 1\\
         $G_{5,18}^{\alpha}$ & $ G_{5,11}^{1}$ & $G_{5,7}^{1,1,1}$  & $\mathbf R^4 \rtimes \Delta((\alpha \pm i)^2)$, $\alpha \neq 0$ & 1 \\
         - & - & $G_{5,19}^{0, \beta}$ & $(\mathfrak {heis} \times \mathbf R) \rtimes \Delta(1,-1,0,\beta)$, $\beta \neq 0$ & 2 \\
          - & $G_{5,19}^{1, \beta}$ & $\mathbf R \times G_{4,5}^{1,\beta}$ & $(\mathfrak {heis} \times \mathbf R) \rtimes \Delta(1,0,1,\beta)$, $\beta \neq 0$ & 2 \\
         - & - & $G_{5,19}^{\alpha, \beta}$ & $(\mathfrak {heis} \times \mathbf R) \rtimes \Delta(1,\alpha-1,\alpha,\beta)$, & 1 \\
         & & &  $(\alpha -1) \beta \neq 0$. & \\
         - & $G_{5,20}^0$ & $\mathbf R \times G_{4,8}$ & $(\mathfrak {heis} \times \mathbf R) \rtimes \Delta(1,-1,0^2) $ & 2 \\
         - & $G_{5,20}^1$ & $G_{5,19}^{1,1}$ & $(\mathfrak {heis} \times \mathbf R) \rtimes \Delta(1,0,1^2) $ & 2 \\
         - & $G_{5,20}^\alpha$ & $G_{5,19}^{\alpha, \alpha}$ & $(\mathfrak {heis} \times \mathbf R) \rtimes \Delta(1,\alpha-1,\alpha^2) $, $\alpha \neq 1$. & 1 \\
         
         \bottomrule
         
    \end{tabular}
    \caption{The simply connected, real, indecomposable, solvable Lie groups of exponential growth and dimension five (to be continued on Table~\ref{scr_expgrowth_ctd}).}
    \label{tab:scr_expgrowth}
\end{table}

\begin{table}[H]
    \centering
    \begin{tabular}{llllc}
    \toprule
     $G$ & $ \rho_0(G) $ & $\rho_1(G)$ & structure of $\mathfrak{g}$ & $\operatorname{conedim}(G)$ \\
     \midrule
     - & $G_{5,21}$ & $G_{5,19}^{2, 1}$ & $(\mathbf R \times \mathfrak {heis}) \rtimes \Delta(1^3,2)$ & 1 \\
         - & - & $G_{5,22}$ & $(\mathbf R \times \mathfrak {heis}) \rtimes \Delta(1,0^2,0)$ & 4  \\
     - &$G_{5,23}^{\beta}$ & $G_{5,19}^{2,\beta}$ &  $(\mathfrak {heis} \times \mathbf R) \rtimes \Delta(1^2,2,\beta)$, $\beta \neq 0$ &   1 \\
         - & $G_{5,24}^\epsilon$ & $G_{5,19}^{2,2}$ & $(\mathfrak{heis} \times \mathbf R) \rtimes \phi_{5,24}^\epsilon$, see Example~\ref{exm:G524}. & 1  \\
         $G_{5,25}^{1, 0}$& - & $\mathbf {Heis} \times A_2$ & $(\mathfrak{heis} \times \mathbf R) \rtimes \Delta(\pm i, 0, 1) $ & 4 \\
         $G_{5,25}^{\beta, \alpha}$& - & $G_{5,19}^{2,\beta/\alpha}$ & $(\mathfrak{heis} \times \mathbf R) \rtimes \Delta( \alpha \pm i, 2\alpha, \beta), \alpha \neq 0 $ & 1 \\
         
     - &
         $S_{5,26}^1$ & $G_{5,19}^{2,2}$ & $(\mathfrak{heis} \times \mathbf R) \rtimes \Delta(1,1,2^2)$ & 1 \\
     $G_{5,26}^{\alpha, \epsilon}$ &
         $S_{5,26}^1$ & $G_{5,19}^{2,2}$ & $(\mathfrak{heis} \times \mathbf R) \rtimes \phi_{5,26}^{\alpha,\epsilon}$, & 1 \\
         & & &  $\alpha > 0$, $\epsilon = \pm 1$; see Example~\ref{exm:G526}. & \\
         - & $G_{5,27}$ & $\mathbf R \times G_{4,5}^{1,1} $ & $(\mathfrak{heis} \times \mathbf R) \rtimes \phi_{5,27}$; see Example~\ref{exm:G527} & 2 \\
         - & $G_{5,28}^1$ & $\mathbf R \times G_{4,5}^{1,1} $ & $(\mathbf R \times \mathfrak{heis}) \rtimes \Delta(1^2, 0, 1)$  & 2 \\
         - & $G_{5,28}^\alpha$ & $G_{5,19}^{\alpha,1}$ & $(\mathbf R \times \mathfrak{heis}) \rtimes \Delta(1^2, \alpha -1, \alpha)$, $\alpha >1$  & 1 \\
         - & $G_{5,29}$ & $G_{5,8}^1$ & $(\mathbf R \times \mathfrak{heis}) \rtimes \Delta(0^2, 1,1)$  & 3 \\
         - & - & $G_{5,30}^{-1}$ & $\mathfrak{fil} \rtimes \Delta(1,-2 , -1, 0)$ & 1 \\
         - & - & $G_{5,30}^0$ & $\mathfrak{fil} \rtimes \Delta(1,-1 ,0, 1)$ & 1 \\
         - & $G_{5,30}^1$ & $\mathbf R \times G_{4,9}^1$  & $\mathfrak{fil} \rtimes \Delta(1,0 ,1, 2) $ & 2  \\
         - & - & $G_{5,30}^\alpha$ & $\mathfrak{fil} \rtimes \Delta(1,\alpha -1 , \alpha, \alpha +1)$ & 1 \\
         - & $G_{5,31}$ & $G_{5,30}^2$ & $\mathfrak{fil} \rtimes \Delta(1^2,2, 3)$ & 1 \\
         - & $G_{5,32}^0$ & $\mathbf R \times G_{4,5}^{1,1}$  & $\mathfrak{fil} \rtimes \Delta(0,1,1,1)$ & 2 \\
         - & $G_{5,32}^\alpha$ & $\mathbf R \times G_{4,5}^{1,1}$ & $\mathfrak{fil} \rtimes \phi_{5,32}^\alpha$; see Example~\ref{exm:G532} & 2 \\
         - & - & $G_{5,33}^{0,\beta}$ & $\mathbf R^3 \rtimes \{ \Delta(0,1,0), \Delta(1,0,\beta) \}$, $\beta \neq 0$. & 2\\
         - & - & $G_{5,33}^{\alpha, \beta}$ & $\mathbf R^3 \rtimes \{ \Delta(0,1,\alpha), \Delta(1,0,\beta) \}$, $\alpha \leqslant \beta$, $\alpha \neq 0$ & 2\\
        
         - & - & $G_{5,34}^\alpha$ & $\mathbf R^3 \rtimes \{ \Delta(\alpha,1,1), \Delta(1,0,1) \}$ & 2 \\
         $G_{5,35}^{\alpha, \beta}$ & - & $G_{5,33}^{\alpha, \beta}$ & $\mathbf R^3 \rtimes \{ \Delta(\alpha,\pm i), \Delta(\beta,1,1) \}$, $\alpha \neq 0$. & 2\\
         $G_{5,35}^{0, \beta}$ & - & $\mathbf R \times G_{4,5}^{1,\beta} $ & $\mathbf R^3 \rtimes \{ \Delta(0,\pm i), \Delta(\beta,1,1) \}$, $\beta \neq 0$ & 2 \\
         - & - & $G_{5,36}$\textsuperscript{($\ddagger$)} & $\mathfrak{heis} \rtimes \{ \Delta(1,0,1), \Delta(-1,1,0) \}$ & 2 \\
         $G_{5,37}$ & & $\mathbf R \times G_{4,9}^1$ & $\mathfrak{heis} \rtimes \{\Delta(1,1,2), \Delta(\pm i, 0)\}$ & 2 \\ 
         & & $G_{5,38}$ & $\mathbf R^2 \rtimes \mathfrak{heis}$; see Example~\ref{exm:G53839} & 3 \\
         $G_{5,39}$ & - & $G_{5,8}^1$ &  see Example~\ref{exm:G53839} & 3 \\
         \bottomrule
    \end{tabular}
    
    \begin{footnotesize}
    \textsuperscript{($\ddagger$)} $G_{5,36}$ is the maximal completely solvable subgroup of the rank two simple group $\operatorname{SL}(3,\mathbf R)$. 
    \end{footnotesize}
    \caption{The simply connected, real, indecomposable, solvable Lie groups of exponential growth and dimension five (started on Table \ref{tab:scr_expgrowth}).}
    \label{scr_expgrowth_ctd}
\end{table}

\begin{table}[H]
    \centering
    
    \begin{tabular}{lllllllll}
        \toprule
        Group & $\rho_1$ & S.S. & SOL & HOM & AW & D. Ex. & $\delta(n)$ & Reason \\
        \midrule
         $G_{3,2}$ & $G_{3,3}$ & $\checkmark$ & $\times$ &  &  &  & $n$ & Hyp \\
        
        $G_{3,3}$ & $G_{3,3}$ &  & &  &  &  & $n$ & $\rho_1$ \\
        
        $G_{3,4}$ & $G_{3,4}$ & $\checkmark$ & $\checkmark$ & & & & $\exp(n)$ & SOL \\
        
        $G_{3,5}^{\alpha>0}$ & $G_{3,5}^\alpha$ & $\checkmark$ & & &  & & $n$ & Hyp \\
        
        $G_{3,5}^{\alpha<0}$ & $G_{3,5}^\alpha$ & $\checkmark$ & $\checkmark$ & &  & & $\exp(n)$ & SOL \\ 
        
        $G_{4,2}^{\alpha>0}$ & $G_{4,5}^{1,\alpha}$ & $\checkmark$ & $\times$ & &  &  & $n$ & Hyp  \\
        
        $G_{4,2}^{\alpha<0}$ & $G_{4,5}^{1,\alpha}$ & $\checkmark$ & $\checkmark$ & &  & & $\exp(n)$ & SOL\\
        
        $G_{4,3}$ & $G_{4,3}$ & $\times$ & $\times$ & $\times$ & $\times$ & $\{\omega_{2,3},\omega_{2,4}\}$ & $n^3$ & $n^3$ D.Ex., GT \\
        
        $G_{4,4}$ & $G_{4,5}^{1,1}$ & $\checkmark$ & $\times$ & & &  & $n$ & Hyp \\
        
        $G_{4,5}^{0<\alpha\leq \beta}$& $G_{4,5}^{\alpha,\beta}$ & $\checkmark$ & $\times$ & & &  & $n$ & Hyp \\

        $G_{4,5}^{\alpha<0,\beta\ne 0}$& $G_{4,5}^{\alpha,\beta}$ & $\checkmark$ & $\checkmark$ & & &  & $\exp(n)$ & SOL \\
        
        $G_{4,7}$ & $G_{4,9}^1$ & $\checkmark$ & $\times$ &  &  &  & $n$ & Hyp \\
        
        $G_{4,8}$ & $G_{4,8}$ & $\checkmark$ & $\checkmark$ &  &  & & $\exp(n)$ & SOL \\
        
        $G_{4,9}^0$ & $\mathbf R\times G_{3,3}$ & $\checkmark$ &  &  & $\times$  &  & $n^2$ & $\rho_1$, K \\
        
        $G_{4,9}^{0<\beta\leq 1}$& $G_{4,9}^\beta$ & $\checkmark$ & $\times$ &  & & & $n$ & HYP \\
        
        $G_{4,9}^{-1\leq\beta<0}$ & $G_{4,9}^\beta$ & $\checkmark$ & $\checkmark$ & & &  & $\exp(n)$ & SOL \\
           \bottomrule
          
    \end{tabular}
     \caption{Indecomposable completely solvable groups of exponential growth and dimension less than $5$ and their Dehn functions.}
     \label{tab:groupsless5prop}
\end{table}

\begin{table}[htbp]

\begin{small}
    \centering
    \begin{tabular}{|lll|lll|}
    \hline
    Group & $\delta(n)$ & Reason & Group & $\delta(n)$ & Reason \\
    \hline
    $G_{5,7}^{\alpha\geq\beta\geq\gamma>0}$ & $n$ & Hyp & $G_{5,23}^{\beta>0}$ & $n$ & Hyp \\
    $G_{5,7}^{0\ne \alpha,0\ne \beta,\gamma<0}$ & $\exp(n)$ & SOL & $G_{5,23}^{\beta<0}$ & $\exp(n)$ & SOL \\
    $G_{5,8}^{-1\leq\gamma<0}$ & $\exp(n)$ & SOL & $G_{5,24}^{\epsilon}$ & $n$ & $\rho_1$ \\
    $G_{5,8}^{0<\gamma\leq 1}$ & $n^3$ & $n^3$ D.Ex. $\{\omega_{51}\}$, GT & $G_{5,25}^{1,0}$ & $n^3$ & $\rho_1=\rho_0$ \\
    $G_{5,9}^{0<\beta\leq \gamma}$ & $n$ & Hyp & $G_{5,25}^{\beta,\alpha,\beta\alpha<0}$ & $\exp(n)$ & $\rho_1$ \\
    $G_{5,9}^{0>\beta,\gamma\ne 0}$ & $\exp(n)$ & SOL & $G_{5,25}^{\beta,\alpha,\beta\alpha>0}$ & $n$ & $\rho_1=\rho_0$ \\
    $G_{5,10}$ & $n^4$ & $n^4$ D.Ex. $\{\omega_{15},\omega_{23}\}$, GT & $S_{5,26}^1$ & $n$ & $\rho_1$ \\
    $G_{5,11}^{\gamma>0}$ & $n$ & Hyp & $S_{5,26}^{\alpha,\epsilon}$ & $n$ & $\rho_1$ \\
    $G_{5,11}^{\gamma<0}$ & $\exp(n)$ & SOL & $G_{5,27}$ & $n^2$ & $\rho_1$, K \\
    $G_{5,12}$ & $n$ & $\rho_1$ & $G_{5,28}^1$ & $n^2$ & $\rho_1$, K \\
    $G_{5,13}^{1>\alpha,0,1}$ & $\exp(n)$ & $\rho_1$ & $G_{5,28}^{\alpha>1}$ & $n$ & $\rho_1$ \\
    $G_{5,13}^{1<\alpha,0,1}$ & $n^2$ & $\rho_1=\rho_0$ & $G_{5,28}^{\alpha<1}$ & $\exp(n)$ & $\rho_1$ \\
    $G_{5,13}^{\alpha\leq \beta,\tau:\alpha<0}$ & $\exp(n)$ & $\rho_1$ & $G_{5,29}$ & $n^3$ & $n^3$ D.Ex. $\{\omega_{12},\omega_{15}\}$, GT \\
    $G_{5,13}^{0<\alpha\leq\beta,\tau}$ & $n$ & $\rho_1$ & $G_{5,30}^{-1}$ & $\exp(n)$ & SOL \\
    $G_{5,14}^{\alpha\ne 0}$ & $n^3$ & $\rho_1=\rho_0$ & $G_{5,30}^{0}$ & $\exp(n)$ & SOL \\
    $G_{5,15}^0$ & $n^3$ & $\rho_1=\rho_0$ & $G_{5,30}^{1}$ & $n^2$ & $\rho_1$, K \\
    $G_{5,15}^{\beta<0}$ & $\exp(n)$ & $\rho_1$ & $G_{5,30}^{1>\alpha\notin\{-1,0\}}$ & $\exp(n)$ & SOL \\
    $G_{5,15}^{\beta>0}$ & $n$ & $\rho_1$ & $G_{5,30}^{1<\alpha}$ & $n$ & Hyp \\
    $G_{5,16}^{0,\tau\ne 0}$ & $n^2$ & $\rho_1=\rho_0$ & $G_{5,31}$ & $n$ & $\rho_1$ \\
    $G_{5,16}^{0>\beta,1}$ & $\exp(n)$ & $\rho_1$ & $G_{5,32}^\alpha$ & $n^2$ & $\rho_1$, K \\
    $G_{5,16}^{0<\beta,1}$ & $n$ & $\rho_1$ & $G_{5,33}^{0,\beta<0}$ & $\exp(n)$ & SOL \\
    $G_{5,17}^{0\ne\tau,0,1}$ & $n^2$ & $\rho_1=\rho_0$ & $G_{5,33}^{0,\beta>0}$ & $n^2$ & A-W, not Hyp \\
    $G_{5,17}^{0\ne\tau,\alpha,\beta:\alpha\beta>0}$ & $n$ & $\rho_1$ & $G_{5,33}^{\alpha < \beta=0}$ & $\exp(n)$ & SOL \\
    $G_{5,17}^{0\ne\tau,\alpha,\beta:\alpha\beta<0}$ & $\exp(n)$ & $\rho_1$ & $G_{5,33}^{0<\alpha,\beta=0}$ & $n^2$ & C-T, not Hyp, K \\
    $G_{5,18}^{\alpha\ne 0}$ & $n$ & $\rho_1$ & $G_{5,33}^{0<\alpha\leq \beta}$ & $n^2$ & A-W, not Hyp \\
    $G_{5,19}^{0,\beta\ne 0}$ & $\exp(n)$ & SOL & $G_{5,33}^{\alpha\leq \beta<0}$ & $n^2$ & C-T, not Hyp, K \\
    $G_{5,19}^{1,\beta<0}$ & $\exp(n)$ & SOL & $G_{5,34}^\alpha$ & $n^2$ & A-W, not Hyp \\
    $G_{5,19}^{1,\beta>0}$ & $n^2$ & $\rho_1$, K & $G_{5,35}^{0,\beta<0}$ & $\exp(n)$ & $\rho_1=\rho_0$ \\
    $G_{5,19}^{\alpha,\beta:(\alpha-1)\beta<0}$ & $\exp(n)$ & SOL & $G_{5,35}^{0,\beta>0}$ & $n^2$ & $\rho_1=\rho_0$ \\
    $G_{5,19}^{\alpha,\beta:(\alpha-1)\beta>0}$ & $n$ & Hyp & $G_{5,35}^{\alpha < \beta=0}$ & $\exp(n)$ & $\rho_1=\rho_0$ \\
    $G_{5,20}^0$ & $\exp(n)$ & SOL & $G_{5,35}^{0<\alpha,\beta=0}$ & $n^2$ & $\rho_1=\rho_0$ \\
    $G_{5,20}^1$ & $n^2$ & $\rho_1$, K & $G_{5,35}^{0<\alpha\leq \beta}$ & $n^2$ & $\rho_1=\rho_0$ \\
    $G_{5,20}^{\alpha:(\alpha-1)\alpha<0}$ & $\exp(n)$ & $\rho_1$ & $G_{5,35}^{\alpha\leq \beta<0}$ & $n^2$ & $\rho_1=\rho_0$ \\
    $G_{5,20}^{\alpha:(\alpha-1)\alpha>0}$ & $n$ & $\rho_1$ & $G_{5,36}$ & $n^2$ & A-W, not Hyp \\
    $G_{5,21}$ & $n$ & $\rho_1$ & $G_{5,37}$ & $n^2$ & $\rho_1=\rho_0$ \\
    $G_{5,22}$ & $n^4$ & $n^4$ D.Ex. $\{\omega_{25},\omega_{34}\}$, GT & $G_{5,38}$ & $n^3$ & $n^3$ D. Ex. $\{ \omega_{35}; \omega_{34} \}$, GT \\
    & & & $G_{5,39}$ & $n^3$ & $\rho_1 = \rho_0$ \\
    \bottomrule
    \end{tabular}
    \end{small}
    \caption{Dehn functions of $5$-dimensional simply connected indecomposable solvable Lie groups.}   \label{tab:Dehn-functions-5}
\end{table}

\subsection{Computations}\label{sec: computation examples}
In this section we explain how to compute the data and criteria described above. We do not give the complete computations, but rather provide full details on a few selected examples in each category. 

\subsubsection{Computing $\rho_1(G)$ for simply connected solvable $G$} Beware that we do not take the same bases as in \cite{Mubarakzyanov,PateraZassenhaus} when working in the Lie algebras: often, the order and the sign of the vectors is changed according to our needs.

\begin{example}\label{exm:G43}
    The Lie algebra of the group $G_{4,3}$ is $\mathbf R^3 \rtimes \Delta(1,0^2)$. Let $(e_1,\ldots e_4)$ be its basis. 
    The derived subalgebra is $C^2 \mathfrak g = \langle e_1, e_2 \rangle$, and the exponential radical is $C^3 \mathfrak g = \langle e_1 \rangle$. 
    The quotient $\mathfrak g_{4,3} / \operatorname{R}_{\exp} \mathfrak g_{4,3}$ is isomorphic to $\mathfrak{heis}$, the exponential radical is split, and the action on it is $\mathbf R$-diagonalizable. So $G_{4,3}$ is in $(\mathcal C_1)$.
\end{example}

\begin{example}
    The Lie algebra of the group $G_{4,7}$ is $\mathfrak{heis} \rtimes \Delta(1^2,2)$, which means that it has a basis $(e_1,e_2,e_3,e_4)$ where $e_1, e_2$ and $e_3$ generate a Heisenberg ideal, with $[e_1,e_2] = e_3$, and that
   \[ \operatorname{ad}_{e_4} = \begin{pmatrix}
        1 & 1 & 0 \\
        0 & 1 & 0 \\
        0 & 0 & 2
    \end{pmatrix} \]
    in the basis $(e_1,e_2,e_3)$. We have $C^3 G_{4,7} = C^2 G_{4,7}$, so the exponential radical is the nilradical.
    To compute $\rho_1(G_{4,7})$ we remove the nilpotent part in the Jordan decomposition of $\operatorname{ad}_{e_4}$, this gives a new Lie algebra law which is that of $\mathfrak{heis} \rtimes \Delta(1,1,2)$. We find that it is $G_{4,9}^1$.
    ($G_{4,7}$ is QI rigid within $(\mathcal C_0)$ by the combination of \cite{CoTesContracting} and \cite{CarrascoSequeira}).
\end{example}

\begin{example}\label{exm:G524}
    The Lie algebra of the group $G_{5,24}^\epsilon$ has a nilradical $\mathfrak n = \mathfrak{heis} \times \mathbf R$, with basis $(e_1,e_2,e_3,e_4)$ such that $[e_1,e_2] = e_3$, and
    \[ \operatorname{ad}_{e_5} = \phi_{5,24}^\epsilon := \begin{pmatrix}
        1 & 1 & 0 & 0 \\
        0 & 1 & 0 & 0 \\
        0 & 0 & 2 & \epsilon \\
        0 & 0 & 0 & 2
    \end{pmatrix}, \quad \epsilon = \pm 1. \]
    (Beware the basis in \cite{Mubarakzyanov} would rather be $(e_2,e_1,-e_3,e_4,-e_5)$). Since $\operatorname{ad}_{e_5}$ is nondegenerate on $\mathfrak n$, the exponential radical is again equal to the nilradical.
    If $\epsilon = 1$ then $\phi_{5,24}^\epsilon$ has Jordan normal form in the given basis, and its diagonal part is that of type $\Delta(1,1,2,2)$. If $\epsilon =-1$, then $\phi_{5,24}^\epsilon$ has its standard Jordan form in the basis $(f_1,f_2,f_3,f_4):=(e_1,e_2,-e_3,e_4)$ and the diagonal part the same. (Note that $[f_1,f_2] = -f_3$, which is why we could not express the structure of $\mathfrak g_{5,24}^\epsilon$ in the table.) In both cases, $\rho_1(\mathfrak g_{5,24}^{\epsilon}) = (\mathfrak {heis} \times \mathbf R) \rtimes \Delta(1,1,2,2) = \mathfrak g_{5,19}^{2,2}$.
\end{example}

\begin{example}\label{exm:G526}

The Lie algebra of the group $G_{5,26}^{\alpha,\epsilon}$ has a basis $(e_1,\ldots, e_5)$ with $[e_1,e_2] = e_3$, $[e_5,e_1] = \alpha e_1 +e_2$, $[e_5,e_2] = \alpha e_2 - e_1$, $[e_5,e_3] = 2\alpha e_3$ and $[e_5,e_4]= 2\alpha e_4 +\epsilon e_3$. Without loss of generality the parameter $\alpha$ is positive\footnote{We might as well define the group for $\alpha < 0$. However, exchanging $e_1$ and $e_2$ while turning $e_3$ and $e_4$ into their opposite we see that $G_{5,26}^{-\alpha, \epsilon} \simeq G_{5,26}^{\alpha, -\epsilon}$; there are no further isomorphism in this family thanks to the invariant given in \cite{PateraZassenhaus}, where $\alpha$ is denoted by $p$.}; set
    \[ (f_1, f_2,f_3,f_4,f_5) = 
     \left( \frac{\epsilon}{\sqrt{\alpha}} e_1, \frac{1}{\sqrt{\alpha}}e_2, \frac{\epsilon}{\alpha} e_3, e_4, \frac{1}{\alpha} e_5 \right). \]
    Then, with $\tau = \alpha^{-1/2}$,
     \[ \left[ \operatorname{ad}_{f_5} \right]_{(f_1,f_2,f_3,f_4)} = \phi_{5,26}^{\alpha, \epsilon} =\begin{pmatrix}
        1 & - \tau & 0 & 0 \\
        \tau & 1 & 0 & 0 \\
        0 & 0 & 2 & \epsilon \\
        0 & 0 & 0 & 2
    \end{pmatrix}, \quad \epsilon = \pm 1. \]
    The derivation $\phi_{5,26}^{\alpha, \epsilon}$ has Jordan type $\Delta(1+i\tau, 2^2)$, so $\rho_0(G_{5,26}^{\alpha, \epsilon})$ is $S_{5,26}^1$.
\end{example}

\begin{example}\label{exm:G527}

    The Lie algebra of $G_{5,27}$ has nonzero brackets $[e_1,e_2] = e_3$, $[e_5,e_3] = e_3$, $[e_5,e_2] = e_2 + e_4$, $[e_5,e_4] = e_3+e_4$. Thus $C^2 G_{5,27} = C^3 G_{5,27} = \langle e_2, e_3, e_4 \rangle$ and this is the exponential radical. The cone dimension is therefore $2$. Moreover, the exponential radical splits, and in the basis $(e_3,e_4,e_2)$, $\operatorname{ad}_{e_5}$ has type $\Delta(1^3)$ while $\operatorname{ad}_{e_1}$ is nilpotent. From this we deduce that $\rho_1(G_{5,27})$ is $\mathbf R \times G_{4,5}^{1,1}$.
\end{example}

\begin{example}\label{exm:G532}
    The Lie algebra $\mathfrak g_{5,32}^\alpha$ has a basis $(e_1, \ldots, e_5)$ where $[e_1,e_2] = e_3, [e_1, e_3] = e_4$ and the matrix of $\operatorname{ad}_{e_5}$ in this basis is
    \[
    [\operatorname{ad}_{e_5}] = \phi_{5,32}^\alpha =
\begin{pmatrix}
   0 & 0 & 0 & 0 \\
   0 & 1 & 0 & 0 \\
   0 & 0 & 1 & 0 \\
   0 & \alpha & 0 & 1
\end{pmatrix}
    \]
   The derived subgroup is $\langle e_2, \ldots,  e_4\rangle$ and this is the exponential radical. $e_1$ acts nilpotently on the exponential radical. Hence the Lie algebra of $\rho_1(G_{5,32}^\alpha)$ is, in the same basis, $\langle e_2, e_3, e_4 \rangle \rtimes \langle e_1, e_5 \rangle$, where $\operatorname{ad}_{e_1} = 0$ and $\operatorname{ad}_{e_5} = 1$. This further splits as $(\langle e_2, e_3, e_4 \rangle \rtimes \langle e_5 \rangle) \times \langle e_1 \rangle$, with the isomorphism type of $G_{4,5}^{1,1} \times \mathbf R$.
\end{example}

\begin{example}[$G_{5,38}$ and $G_{5,39}$]\label{exm:G53839}
The Lie algebra $\mathfrak g_{5,38}$ has a two-dimensional abelian ideal $\mathfrak r$, generated by the basis element $(e_1,e_2)$, and splits as a semidirect product $\mathfrak r \rtimes \mathfrak{heis}$, where a section of $\mathfrak{heis}$ is generated by $(e_3,e_4,e_5)$, where $[e_5,e_4] = e_3$, 
\[ \operatorname{ad}_{e_4} = \begin{pmatrix}
    1 & 0 \\ 0 & 0 
\end{pmatrix} \quad \text{and} \quad \operatorname{ad}_{e_5} = \begin{pmatrix}
    0 & 0 \\ 0 & 1
\end{pmatrix} \]
in the basis $(e_1,e_2)$ of $\mathfrak r$.  
The nilradical is $\mathfrak n = \mathfrak r + \mathbf Re_5$, it is not split.
We compute that $C^2 \mathfrak g_{5,38} = \mathfrak n$ and $C^3 \mathfrak g_{5,38} = C^4 \mathfrak g_{5,38} = \mathfrak r$, so that $\mathfrak r$ is the exponential radical. The action of $\operatorname{span}(e_3,e_4,e_5)$ being diagonal, $\mathfrak g_{5,38}$ is in $(\mathcal C_1)$.

Let us now turn to $\mathfrak g_{5,39}$
The structure of the Lie algebra $\mathfrak g_{5,39}$ is also $\mathfrak r \rtimes {\mathfrak{heis}}$ but this time the adjoint action of $e_4$ and $e_5$ on $\mathfrak r$  are, instead,
\[ \operatorname{ad}_{e_4} = \begin{pmatrix}
    1 & 0 \\ 0 & 1 
\end{pmatrix} \quad \text{and} \quad \operatorname{ad}_{e_5} = \begin{pmatrix}
    0 & 1 \\  -1 & 0
\end{pmatrix} \]
in the basis $(e_1,e_2)$ of $\mathfrak r$. Again, the nilradical $\mathfrak n = \mathfrak r + \mathbf Re_3$ is not split, and $\mathfrak r$ is the exponential radical. However the action of $e_5$ on the nilradical has a purely imaginary type, so that, $\rho_0(\mathfrak g_{5,39}) = \mathfrak r \rtimes \mathfrak{heis}$ with $e_5$ centralizing $\mathfrak r$. Further, $\rho_0(\mathfrak g_{5,39})$ admits the following description: $\operatorname{span}(e_1,e_2,e_3,e_5)$ is an abelian ideal and, in the basis $(-e_3,e_5,e_1,e_2)$, $\operatorname{ad}_{e_{4}}$ has matrix of type $\Delta(0^2,1,1)$. Consequently, $\rho_0(\mathfrak g_{5,39})=\mathfrak g_{5,8}^{1}$.
\end{example}

\subsubsection{Standard Solvability.} \label{sec: standard solvability}
We need to check whether $G$ splits as $G=U\rtimes A$, where $U$ is the exponential radical of $G$, $A$ is abelian and the action of $A$ on $U/[U,U]$ has no fixed points (Proposition~\ref{prop: UA standard solvable can take u to be ExpRad}). When $A$ is $1$-dimensional this becomes a very easy task. 
In the general case, in the presence of an abelian complement $A$ to $U$ in $G$, it is quite straight forward to check the eigenvalues of the action of $A$ on $U/[U,U]$. When $G/U$ is non-abelian, this rules out the possibility of $U$ having an abelian complement in $G$.

\begin{example}\label{exm: non Standard Solvable}
    Let us check that $G_{5,8}^\gamma$ is not standard solvable. The nonzero Lie brackets are:
    $$[e_5,e_2]=e_1, [e_5,e_3]=e_3, [e_5,e_4]=\gamma e_4;$$

    The exponential radical $\mathfrak{u}$ is generated by $\{e_3,e_4\}$ and is isomorphic to $\mathbf{R}^2$. Any complement of this must be isomorphic to the quotient of $\mathfrak{g}$ by $\mathfrak{u}$, which is the Heisenberg algebra. In particular there is no abelian complement, and $G_{5,8}^\gamma$ is not standard solvable.
\end{example}

\subsubsection{Azencott-Wilson criterion} The Azencott-Wilson criterion from \cite{azencott1976homogeneous} states that a completely solvable Lie group $G$ admits a left-invariant nonpositively curved Riemannian metric if the following conditions on its Lie algebra $\mathfrak g$ are met: 
\begin{enumerate}
    \item $\mathfrak{g}=\mathfrak{n}\oplus\mathfrak{a}$, where $\mathfrak{n}$ is its nilradical and $\mathfrak{a}$ is abelian. 
    \item For every root $\alpha$ and $H\in\mathfrak{a}$ with $\alpha(H)=0$, $\operatorname{ad}_{H \mid \mathfrak n^\alpha}$ is semisimple. 
    \item The set of roots that are different from $0$ lie in an open half-space.
    \item The zero weight space is central in $\mathfrak{n}$.
    \item For each root $\alpha$, let
    $$\mathfrak{n}_\alpha^0=\{X\in \mathfrak{n}_\alpha\}\mid [X,\mathfrak{n}_\beta]=0\  \text{whenever $\gamma$ is a root linearly independent of $\alpha$}\}.$$
    Then for all $\alpha$, the space $\mathfrak{n}_\alpha^0$ is $\mathfrak{a}$-invariant and admits an $\mathfrak{a}$-invariant complement on which $\mathfrak{a}$ acts semisimply.
\end{enumerate}

\begin{example}
   We show that $G_{5,36}$ admits the Azencott-Wilson criterion. We remark that this group is the Borel subgroup of $\operatorname{SL_3(\mathbf{R})}$ so the fact that it acts on a nonpositively curved space is well known. We give this example in detail only to familiarize the reader with the Azencott-Wilson criterion.

   The Lie algebra $\mathfrak{g}_{5,36}$ is given by the brackets: 
   $$[e_1,e_2]=e_3, [e_4,e_1]=e_1, [e_4,e_3]=e_3,[e_5,e_1]=-e_1,[e_5,e_2]=e_2$$
\end{example}
The nilradical of $\mathfrak{g}_{5,36}$ is $\mathfrak{n}=\langle e_1,e_2,e_3\rangle$, isomorphic to $\mathfrak{heis}$. It has a natural abelian complement in $\mathfrak{a}:=\langle e_4,e_5\rangle$. The action of all elements of $\mathfrak{a}$ is semisimple, so condition $2$ is met. The set of roots of $\mathfrak{a}$ is $\{(1,-1),(0,1)(1,0)\}$, all lying in an open half-space of $\mathbf{R}^2$. The $0$-root space is trivial and in particular central in $\mathfrak{n}$. Condition $(5)$ is the most involved: For the root $\alpha=(1,-1)$, $\mathfrak{n}_\alpha=\langle e_1 \rangle$. Since $[e_1,e_2]\ne 0$, the space 
$\mathfrak{n}_\alpha^0=\{0\}$. Obviously this space is $\mathfrak{a}$-invariant, its complement in $\mathfrak{n}$ is the whole $\mathfrak{n}$ which is $\mathfrak{a}$-invariant. Finally, $\mathfrak{a}$ acts semisimply on $\mathfrak{n}$, so condition $(5)$ is met for the root $(-1,1)$. The same argument works for the root $(0,1)$, as $\mathfrak{n}_{(0,1)}^0=\{0\}$. For the root $(1,0)$ we have $\mathfrak{n}_{(1,0)}^0=\langle e_3\rangle$. But also this space is easily seen to satisfy the requirements of condition $(5)$. We conclude that $\mathfrak{g}_{5,36}$ admits the Azencott-Wilson criterion. 
\begin{example}
    We show that the group $G_{5,30}^1$ does not admit the Azencott-Wilson criterion. Notice that this group does have Dehn function $n^2$, and could possibly act on a nonpositively curved space.

    The Lie algebra is defined by the filiform brackets $[e_1,e_2]=e_3, [e_1,e_3]=e_4$, and non-zero action of $e_5$ given by $[e_5,e_1]=e_1, [e_5,e_3]=e_3, [e_5,e_4]=2e_4$. We see that the nilradical is spanned by $\{e_1,e_3,e_4\}$ and that the complement spanned by $\{e_2,e_5\}$ is abelian. However, $e_2$ acts on the nilradical with ordered basis $(e_3,e_1,e_4)$ via the following matrix: 
    \[
\begin{pmatrix}
0 & 1 & 0 \\
0 & 0 & 0 \\
0 & 0 & 0
\end{pmatrix}.
\]
Considering the root $\alpha$ such that $\alpha(e_5) = 1$ and $\alpha(e_2)=0$, we observe that $e_2$ does not act semisimply on the space $\mathfrak n^\alpha = \langle e_3, e_1 \rangle$. This violates condition (2) of the criterion. 
\end{example}

\subsubsection{{\rm SOL} obstruction} 
Propositions 4.C.3 and 4.9.D of~\cite{CoTesDehn} state that a group admits the SOL obstruction if and only if its exponential radical admits two quasi-opposite principal weights. By weight we  mean elements of $\operatorname{Hom}(G/U,\mathbf{R})$ with non-zero eigenspaces; weights are principal if they are weights of the action on $U/[U,U]$.

\begin{example}
    We show that $G_{5,23}^{\beta<0}$ admits the SOL obstruction. The defining Lie brackets are
    $$[e_1,e_2]=e_3,[e_5,e_1]=e_1,[e_5,e_2]=e_1+e_2,[e_5,e_3]=2e_3,[e_5,e_4]=\beta e_4$$
    The Lie algebra of the exponential radical is $\mathfrak{u}= \langle e_1,e_2,e_3,e_4\rangle$, a direct product of $\mathfrak{heis}$ and $\mathbf{R}$. Therefore $\mathfrak{u}/[\mathfrak{u},\mathfrak{u}]=\langle e_1,e_2,e_4\rangle$, which is $\mathbf{R}^3$. The action of $\mathfrak{a}=\langle e_5\rangle$ on this space is given by the roots $(1),(1),(\beta)$. Since $\beta<0$, we see that $0$ is in the convex hull of the principal roots, and conclude $G_{5,23}^{\beta<0}$ admits the SOL obstruction. 
\end{example}

\subsubsection{$2$-homological obstruction}
Let $\mathfrak{u}$ be the Lie algebra of the exponential radical $U$, and consider the action of $G/U$ on $U$. This action extends to an action on $H_2(\mathfrak{u})$, the second homology of $\mathfrak{u}$, defined by linearly extending the action on $2$-vectors given by $t.(v_1\wedge v_2)=t.v_1\wedge v_2+v_1\wedge t.v_2$. 
The $2$-homological obstruction states that if this action on $\operatorname{H}_2(\mathfrak{u})$ has a non-trivial zero eigenspace, then $G$ has exponential Dehn function. The computation of this condition is algorithmic. We give one example to manifest it.

\begin{example}
    We show the group $G_{5,29}$ does not have the $2$-homological obstruction. The defining brackets are: 
    $$[e_2,e_3]=e_4,[e_5,e_2]=e_1,[e_5,e_3]=e_3,[e_5,e_4]=e_4$$
    The Lie algebra of the exponential radical is $\mathfrak{u}:=\langle e_3,e_4\rangle$ which is isomorphic to $\mathbf{R}^2$. 
    The second homology group of $\mathbf{R}^2$ is one dimensional generated by $e_3\wedge e_4$. 
    Since it is $1$-dimensional, it is enough to find one element of $\mathfrak a$ which acts non-trivially on it. 
    Indeed, $e_5.(e_3\wedge e_4)=2e_3\wedge e_4$. 
    We conclude that $G_{5,29}$ does not admit the $2$-homological obstruction. 
\end{example}

We remark that in \cite[Section 1.5.3]{CoTesDehn}, Cornulier and Tessera give an example of a group that admits the $2$-homological obstruction but not the SOL obstruction. We note that in all groups we checked (i.e.\ up to dimension $5$) there is no such group.

\subsubsection{Generalized tame groups}
Assume $G=U\rtimes N$. An element $c\in N$ acts as a \emph{compaction} on $U$ if there is a compact set $\Omega\subset U$ such that  for every compact subset $K\subset U$ there is $n\geq 0$ such that $c^n(K)\subset \Omega$.

\begin{definition}[\cite{CoTesDehn}, Definition~6.E.1]
    A locally compact group $G$ is generalized tame if it has a semi-direct product decomposition $G=U\rtimes N$ where some element $c$ of $N$ acts on $U$ as a compaction, and $N$ is nilpotent and compactly generated. 
\end{definition}

\begin{example}
Consider $G_{5,8}^\gamma$ for $\gamma\in(0,1]$. The defining brackets are: 
$$[e_5,e_1]=0, [e_5,e_2]=e_1,[e_5,e_3]=e_3, [e_5,e_4]=\gamma e_4$$

The exponential radical is $\langle e_3,e_4\rangle$, and the quotient by it is $\langle e_1,e_2,e_5\rangle$, isomorphic to the Heisenberg group where $e_1$ is central. The element $\exp(e_5)$ acts as a compaction on the exponential radical. The Dehn function of the Heisenberg group is cubic, therefore $G$ has cubic Dehn function. 

\end{example}

\subsubsection{Central Extensions} See example in Section~\ref{sec: example of central extension}.

\subsubsection{Computing $\operatorname{Kill}(\operatorname{R}_{\exp} \mathfrak g)_0$}

When a non-hyperbolic standard solvable group $G$ does not have the SOL nor the 2-homological obstruction, \cite[Theorem F]{CoTesDehn} gives a sufficient condition for the Dehn function of $G$ to be quadratic. 
This condition is given by the vanishing of the zero weight submodule in the Killing module of $\operatorname{R}_{\exp} \mathfrak g$, that is, in the quotient of the symmetric square of $\operatorname{R}_{\exp} \mathfrak g$ by the submodule spanned by elements of the form $[x,y] \odot z - x \odot [y,z]$ for $x,y,z \in \operatorname{R}_{\exp} \mathfrak g$. When the exponential radical is abelian, this is not a proper quotient and we will still denote the elements of the Killing module by their representatives in the symmetric square.

\begin{example}
    Let us prove that $\operatorname{Kill}(\operatorname{R}_{\exp} \mathfrak g_{4,9}^{0})_0= 0$. 
    The nonzero Lie brackets in $\mathfrak g_{4,9}^{0}$ are $[e_1,e_2] = e_3$, $[e_4,e_1] = e_1$. The exponential radical is spanned by $e_1$ and $e_3$; it is abelian, so that $\operatorname{Kill}(\operatorname{R}_{\exp}\mathfrak g_{4,9}^0)$ is three dimensional, spanned by $e_1 \odot e_1$, $e_1\odot e_3=e_3 \odot e_1$, and $e_3\odot e_3$. Using $e_{i} \cdot (e_j \odot e_k)=[e_i,e_j]\odot e_k + e_j \odot [e_i, e_k]$ for $i\in \{2,4\}$ and $j,k \in \{1,3\}$, we obtain that 
    \begin{align*}
        & e_2 \cdot (e_1 \odot e_1) = - 2 e_1 \odot e_3 
        & e_4 \cdot (e_1 \odot e_1) = 2 e_1 \odot e_1 \\
        & e_2 \cdot (e_1 \odot e_3) = - e_3 \odot e_3 
        & e_4 \cdot (e_1 \odot e_3) = 2e_1 \odot e_3  \\
        & e_2 \cdot (e_3 \odot e_3) = 0 
        & e_4 \cdot (e_3 \odot e_3) = 2e_3 \odot e_3,
    \end{align*}
    so that $\operatorname{Kill}(\operatorname{R}_{\exp} \mathfrak g_{4,9}^{0})_0= 0$.
\end{example}

\begin{example}[An example with non-abelian exponential radical]
    Let us check that $\operatorname{Kill}(\operatorname{R}_{\exp} \mathfrak g_{5,30}^{1})_0= 0$. 
    The nonzero Lie brackets are $[e_1,e_2] = e_3$, $[e_1,e_3] = e_4$, $[e_5,e_1] = e_1$, $[e_5,e_3]=e_3$, $[e_5,e_4] = 2e_4$. The exponential radical is spanned by $e_1,e_3,e_4$; it is non abelian, and the Killing module is a quotient of its symmetric square by the submodule spanned by all the symmetric tensors involving $e_4$; indeed,
    \begin{align*}
        [e_1,e_3] \odot e_3 - e_1 \odot [e_3,e_4] & = e_4 \odot e_4; \\
        [e_1,e_3] \odot e_3 - e_1 \odot [e_3,e_3] & = e_3 \odot e_4; \\
        [e_1,e_3] \odot e_1 - e_1 \odot [e_3,e_1] & = e_1 \odot e_4 - (e_1 \odot - e_4) = 2 e_1 \odot e_4.
    \end{align*}
    Now 
    \begin{align*}
        & e_2 \cdot [e_1 \odot e_1] = - 2 [e_1 \odot e_3] 
        & e_5 \cdot [e_1 \odot e_1] = 2 [e_1 \odot e_1] \\
        & e_2 \cdot [e_1 \odot e_3] = - [e_1 \odot e_3]
        & e_5 \cdot [e_1 \odot e_3] = 2 [e_1 \odot e_3] \\
        & e_2 \cdot [e_3 \odot e_3] = 0
        & e_5 \cdot [e_3 \odot e_3] = 2 [e_3 \odot e_3],
    \end{align*}
    finishing the proof that $\operatorname{Kill}(\operatorname{R}_{\exp} \mathfrak g_{5,30}^{1})_0= 0$. The group $G_{5,30}^1$ is standard solvable, and sublinear bilipschitz equivalent to $G_{4,9}^1 \times \mathbf R$, so that its Dehn function was a priori between $n^2$ and $n^2 \log^4 n$ by Corollary~\ref{cor: logSBE implies log distortion in Dehn functions}; the computation of the zero weight subspace in the Killing module above raises this indetermination, and the Dehn function of $G_{5,30}^1$ is quadratic. 
\end{example}

\begin{example}[An example with parameters]
\label{exm:computing-killing-g-5-33}
    Let us compute $\operatorname{Kill}(\operatorname{R}_{\exp} \mathfrak g^{\alpha, \beta}_{5,33})_0$ depending on $\alpha$ and $\beta$.
    The nonzero Lie brackets are as follows: 
    \begin{equation*}
        [e_4,e_1] = 0,\,  [e_5,e_1] = e_1,\, 
        [e_4,e_2] = e_2,\,   [e_5,e_2] = 0,\,  
        [e_4,e_3] = \alpha e_3,\,  [e_5,e_3] = \beta e_3  .
    \end{equation*}
    We can assume that $(\alpha, \beta) \neq (0,0)$, otherwise $\langle e_3 \rangle$ becomes a direct factor, and $\alpha \leqslant \beta$ without loss of generality.
    Since $\mathfrak u = \operatorname{R}_{\exp} \mathfrak g^{\alpha, \beta}_{5,33} = \operatorname{span}(e_1,e_2,e_3)$ is abelian, $\operatorname{Kill}(\mathfrak u) = \mathfrak u \odot \mathfrak u$.
    We compute that
    \begin{align*}
        & e_4 \cdot (e_1 \odot e_1) = 0 
        & e_4 \cdot (e_2 \odot e_2) = 2 e_2 \odot e_2 \\
        & e_4 \cdot (e_1 \odot e_2) = e_1 \odot e_2
        & e_4 \cdot (e_2 \odot e_3) = (1+\alpha) e_2 \odot e_3 \\
        & e_4 \cdot (e_1 \odot e_3) = \alpha e_1 \odot e_3 
        & e_4 \cdot (e_3 \odot e_3) = 2\alpha e_3 \odot e_3 
    \end{align*}
    and similarly (by exchanging $e_1$ and $e_2$, $\alpha$ and $\beta$)
    \begin{align*}
        & e_5 \cdot (e_1 \odot e_1) = 2 e_1 \odot e_1 &  e_5 \cdot (e_2 \odot e_2) = 0  \\
        & e_5 \cdot (e_1 \odot e_2) = e_1 \odot e_2 &   e_5 \cdot (e_2 \odot e_3) = \beta e_2 \odot e_3  \\
        & e_5 \cdot (e_1 \odot e_3) = (1+\beta) e_1 \odot e_3 &   e_5 \cdot (e_3 \odot e_3) = 2\beta e_3 \odot e_3  .
    \end{align*}
    Thus we get that 
    \begin{equation*}
        \operatorname{Kill}(\operatorname{R}_{\exp} \mathfrak g^{\alpha, \beta}_{5,33})_0 =
        \begin{cases}
        \langle e_2 \odot e_3 \rangle & (\alpha,\beta) = (-1,0) \\
        0 & \text{otherwise}.
        \end{cases}
    \end{equation*}
    (The case $(\alpha, \beta) = (-1,-1)$ is treated in \cite[1.5.2]{CoTesDehn}.)
    Note that when $\alpha>0$, the fact that $\delta_G(n) \asymp n^2$ follows already from the Azencott-Wilson criterion.
\end{example}

\bibliographystyle{alpha}
\bibliography{gabriel}

\end{document}